\documentclass{article} %
\usepackage{nips15submit_e,times}

\usepackage{graphicx, amsmath, amssymb,amsthm}
\usepackage[breaklinks]{hyperref} %
\usepackage{url}
\usepackage{bm,color}
\usepackage{mathdots}

\usepackage[numbers]{natbib} %
\usepackage{multibib} %
\newcites{sup}{Supplementary References}

\usepackage{wrapfig}
\usepackage{multicol}

\usepackage{algcompatible} %
\usepackage{algorithm}
\newlength{\Sfloatsep}
\newlength{\Stextfloatsep}
\newlength{\Sintextsep}
\definecolor{darkgreen}{rgb}{0,.4,.2}
\definecolor{darkblue}{rgb}{.1,.2,.6}
\definecolor{brightblue}{rgb}{0,0.6,0.8}
\hypersetup{
  colorlinks=true,
  linkcolor=darkblue,
  citecolor=darkgreen,
  filecolor=darkblue,
  urlcolor=darkbluea
}

\newtheorem*{rep@theorem}{\rep@title}
\newcommand{\newreptheorem}[2]{%
\newenvironment{rep#1}[1]{%
 \def\rep@title{#2 \ref{##1}}%
 \begin{rep@theorem}}%
 {\end{rep@theorem}}}

\newreptheorem{lemma}{Lemma'}
\newreptheorem{proposition}{Proposition'}
\newreptheorem{theorem}{Theorem'}
\newreptheorem{remark}{Remark'}

\newtheorem{definition}{Definition}
\newtheorem{theorem}[definition]{Theorem}

\newtheorem{lemma}[definition]{Lemma}

\newtheorem{remark}[definition]{Remark}

\DeclareMathOperator*{\conv}{conv}
\DeclareMathOperator{\diam}{diam}

\DeclareMathOperator*{\argmin}{\arg\min}
\DeclareMathOperator*{\argmax}{\arg\max}

\providecommand{\norm}[1]{\left\lVert#1\right\rVert}

\newcommand{\R}{\mathbb{R}}
\newcommand{\X}{\mathcal{X}}

\newcommand{\domain}{\mathcal{M}} %
\newcommand{\stepsize}{\gamma}
\newcommand{\stepmax}{\stepsize_{\textnormal{\scriptsize max}}} %
\newcommand{\stepbound}{\stepsize^\textnormal{\scriptsize B}} %
\newcommand{\FW}{{\hspace{0.05em}\textnormal{FW}}}
\newcommand{\PFW}{{\hspace{0.05em}\textnormal{PFW}}}
\newcommand{\away}{{\hspace{0.06em}\textnormal{\scriptsize A}}}
\newcommand{\Cf}{C_{\hspace{-0.08em}f}}

\newcommand{\CfAFW}{C_{\hspace{-0.08em}f}^\away}

\newcommand{\strongConvAFW}{\mu_{\hspace{-0.08em}f}^\away}

\newcommand{\dirW}{\mathop{dirW}}
\newcommand{\x}{\bm{x}}
\newcommand{\y}{\bm{y}}
\newcommand{\z}{\bm{z}}
\newcommand{\s}{\bm{s}}
\newcommand{\dd}{\bm{d}}

\newcommand{\vv}{\bm{v}} %

\DeclareMathOperator*{\lmo}{LMO_{\!\Vertices}}

\newcommand{\VVertices}{\mathcal{V}}
\newcommand{\Vertices}{\mathcal{A}} %
\newcommand{\Coreset}{\mathcal{S}}
\renewcommand{\S}{\mathcal{S}}
\renewcommand{\aa}{\bm{\alpha}}
\renewcommand{\r}{\bm{r}}
\newcommand{\PdirW}{\mathop{PdirW}}
\newcommand{\PWidth}{\mathop{PW\!idth}}
\newcommand{\innerProd}[2]{\left\langle #1 , #2 \right\rangle}
\newcommand{\innerProdCompressed}[2]{\langle #1 , #2 \rangle}

\newcommand{\Kface}{\mathcal{K}}

\newcommand{\A}{\bm{A}}
\newcommand{\B}{\bm{B}}
\newcommand{\bv}{\bm{b}}
\newcommand{\err}{\bm{e}} %

\newcommand{\xdim}{d}  %
\newcommand{\ydim}{p}  %

\newcommand{\strongConvGeneralized}{\tilde{\mu}_{\hspace{-0.08em}f}}

\newcommand{\0}{\mathbf{0}} %

\newcommand{\ignore}[1]{}%

\newcommand{\remove}[1]{} %

\title{On the Global Linear Convergence \\ of Frank-Wolfe Optimization Variants}

\author{
Simon Lacoste-Julien \\
INRIA - SIERRA project-team\\
{\'E}cole Normale Sup{\'e}rieure, Paris, France \\
\And
Martin Jaggi \\
Dept. of Computer Science \\
ETH Z{\"u}rich, Switzerland \\
}

\nipsfinalcopy %

\begin{document}

\maketitle
\vspace{-2mm}

\begin{abstract}\vspace{-2mm}
The Frank-Wolfe (FW) optimization algorithm has lately re-gained popularity
thanks in particular to its ability to nicely handle the structured
constraints appearing in machine learning applications. However, its
convergence rate is known to be slow (sublinear) when the solution lies at
the boundary. A simple less-known fix is to add the possibility to take `away
steps' during optimization, an operation that importantly \emph{does not}
require a feasibility oracle. %
In this paper, we highlight and clarify several variants of the Frank-Wolfe
optimization algorithm that have been successfully applied in practice: 
away-steps FW, pairwise FW, fully-corrective FW and Wolfe's minimum norm
point algorithm, and prove for the first time that they all enjoy global
linear convergence, under a weaker condition than strong convexity of the objective.
The constant in the convergence rate has an elegant interpretation as the product
of the (classical) condition number of the function with a novel geometric
quantity that plays the role of a `condition number' of the constraint set. 
We provide pointers to where these algorithms have made a difference in
practice, in particular with the flow polytope, the
marginal polytope and the base polytope for submodular optimization.

\end{abstract}

The Frank-Wolfe algorithm \citep{Frank:1956vp} (also known as
\emph{conditional gradient}) is one of the earliest existing methods for
constrained convex optimization, and has seen an impressive revival recently
due to its nice properties compared to projected or proximal gradient
methods, in particular for sparse optimization and machine learning
applications.

On the other hand, the classical projected gradient and proximal methods have
been known to exhibit a very nice adaptive acceleration property, 
namely that the the convergence rate becomes linear for strongly convex
objective, i.e. that the optimization error of the same algorithm after $t$
iterations will decrease geometrically with $O((1-\rho)^{t})$ instead of the
usual $O(1/t)$ for general convex objective functions.
It has become an active research topic recently whether such an acceleration
is also possible for Frank-Wolfe type methods.

\vspace{-2mm}
\paragraph{Contributions.}
We clarify several variants of the Frank-Wolfe algorithm and show that they all converge linearly for any strongly convex function optimized over a polytope domain, with a constant bounded away from zero that only depends on the geometry of the polytope. 
Our analysis does \emph{not} depend on the location of the true optimum with
respect to the domain, which was a disadvantage of earlier existing results
such as \citep{Wolfe:1970wy,Guelat:1986fq,Beck:2004jm}, and the newer work of \citep{Pena:2015ta},
as well as the line of work of 
\citep{Ahipasaoglu:2008il,Kumar:2010ku,Nanculef:2014bj} which rely on
Robinson's condition \citep{Robinson:1982ii}.
Our analysis yields a weaker
sufficient condition than Robinson's condition; in particular we can have
linear convergence even in some cases when the function has more than one
global minima, and is not globally strongly convex. The constant also naturally separates as the product of the condition number of the function with a novel notion of condition number of a polytope, which might have applications in complexity theory.

\vspace{-2mm}
\paragraph{Related Work.}
For the classical Frank-Wolfe algorithm, 
\citep{Beck:2004jm} showed a linear rate for the special case of quadratic
objectives when the optimum is in the strict interior of the domain, 
a result already subsumed by the more general~\citep{Guelat:1986fq}.
The early work of~\citep{Levitin:1966gf} %
showed linear convergence for \emph{strongly convex constraint sets},
under the strong requirement that the gradient norm is not too small
(see~\citep{Garber:2015vq} for a discussion).
The away-steps variant of the Frank-Wolfe algorithm, that can also remove
weight from `bad' atoms in the current active set, was proposed in
\citep{Wolfe:1970wy}, and later also analyzed in~\citep{Guelat:1986fq}.
The precise method is stated below in Algorithm~\ref{alg:AFW}.
\citep{Guelat:1986fq} showed a (local) linear convergence rate on polytopes,
but the constant unfortunately depends on the
distance between the solution and its relative boundary, a
quantity that can be arbitrarily small.
More recently, \citep{Ahipasaoglu:2008il,Kumar:2010ku,Nanculef:2014bj} have
obtained linear convergence results in the case that the optimum solution
satisfies Robinson's condition \citep{Robinson:1982ii}.
In a different recent line of work, \citep{Garber:2013vl,%
Lan:2013um} have
studied a variation of FW %
that repeatedly moves mass from the worst vertices to the standard FW vertex until
a specific condition is satisfied, yielding a linear rate on strongly convex functions. 
Their algorithm requires the knowledge of several constants though, and
moreover is not adaptive to the best-case scenario, unlike the Frank-Wolfe
algorithm with away steps and line-search. 
None of these previous works was shown to be affine invariant, and most
require additional knowledge about problem specific parameters. %

\vspace{-2mm}
\paragraph{Setup.}
We consider general constrained convex optimization problems of the form:
\begin{equation}\label{eq:optGenConvex}
	\min_{\x \in \domain} \, f(\x) \ , \quad\quad  \domain = \conv({\Vertices}), 
	\qquad \text{with only access to: } \,\, \lmo(\r) \in  \argmin_{\x\in \Vertices} \innerProdCompressed{\r}{\x}  ,\vspace{-2mm}
\end{equation}
where $\Vertices \subseteq \R^{\xdim}$ is a \emph{finite} set of vectors that
we call \emph{atoms}.\footnote{The atoms \emph{do not} have to be extreme
points (vertices) of $\domain$.} 
We assume that the function~$f$ is
$\mu$-strongly convex with $L$-Lipschitz continuous gradient over~$\domain$. We also
consider weaker conditions than strong convexity for~$f$ in
Section~\ref{sec:nonStronglyConvex}.
As~$\Vertices$ is finite, $\domain$ is a
(convex and bounded) polytope. The methods that we consider in this paper
only require access to a \emph{linear minimization oracle} $\lmo(.)$
associated with the domain~$\domain$ through a generating set
of atoms~$\Vertices$.
This oracle is defined as to return a minimizer of a linear subproblem
over~$\domain=\conv(\Vertices)$,
for any given direction $\r \in \R^\xdim$.\footnote{%
All our convergence results can be carefully extended to approximate linear
minimization oracles with multiplicative approximation guarantees; we state
them for exact oracles in this paper for simplicity.%
}
\vspace{-3mm}

\paragraph{Examples.} Optimization problems of the form
\eqref{eq:optGenConvex} appear widely in machine learning and signal
processing applications. The set of atoms $\Vertices$ can represent
combinatorial objects of arbitrary type. 
Efficient linear minimization oracles often exist in the form of dynamic
programs or other combinatorial optimization approaches. As an example from
tracking in computer vision, $\Vertices$ could be the set of integer flows on a
graph~\citep{Joulin:2014uw,Chari:2015vo}, where $\lmo$ can be efficiently implemented by a
minimum cost network flow algorithm. In this case, $\domain$ can also be
described with a polynomial number of linear inequalities. But in other
examples, $\domain$ might not have a polynomial description in terms of
linear inequalities, and testing membership in $\domain$ might be much more
expensive than running the linear oracle. This is the case when optimizing
over the \emph{base polytope}, an object appearing in submodular function
optimization~\citep{Bach:2013et}. There, the $\lmo$ oracle is a simple greedy
algorithm. Another example is when $\Vertices$ represents the possible
consistent value assignments on cliques of a Markov random field (MRF);
$\domain$ is the \emph{marginal
polytope}~\citep{wainwright:2008:variational}, where testing membership is
NP-hard in general, though efficient linear oracles exist for some special
cases~\citep{Kolmogorov:2004:graphCut}. Optimization %
over the marginal polytope appears for example in structured SVM
learning~\citep{LacosteJulien:2013ue} and variational %
inference~\citep{Krishnan:2015ws}.\vspace{-2mm}

\paragraph{The Original Frank-Wolfe Algorithm.}
The Frank-Wolfe (FW) optimization algorithm~\citep{Frank:1956vp}, also known
as \emph{conditional gradient}~\citep{Levitin:1966gf}, is particularly suited
for the setup~\eqref{eq:optGenConvex} where $\domain$ is only accessed
through the linear minimization oracle.
It works as follows:
At a current iterate~$\x^{(t)}$, the algorithm finds a feasible search atom
$\s_t$ to move towards by minimizing the linearization of the objective
function $f$ over $\domain$ (line~3 in Algorithm~\ref{alg:AFW}) -- this is
where the linear minimization oracle $\lmo$ is used.
The next iterate $\x^{(t+1)}$ is then obtained by doing a line-search on $f$
between $\x^{(t)}$ and $\s_t$ (line~11 in Algorithm~\ref{alg:AFW}). One
reason for the recent increased popularity of Frank-Wolfe-type algorithms is
the sparsity of their iterates: in iteration $t$ of the algorithm, the
iterate can be represented as a sparse convex combination of at most $t+1$
atoms $\Coreset^{(t)}\subseteq \Vertices$ of the domain~$\domain$, which we
write as $\x^{(t)}= \sum_{\vv \in \Coreset^{(t)}} \alpha^{(t)}_{\vv} \vv\vspace{-0.3mm}$. 
We write $\Coreset^{(t)}$ for the \emph{active set}, containing the
previously discovered search atoms $\s_r$ for $r < t$ that have non-zero
\emph{weight} $\alpha^{(t)}_{\s_r} > 0$ in the expansion (potentially also
including the starting point~$\x^{(0)}$).
While tracking the active set $\Coreset^{(t)}$ is not necessary for the
original FW algorithm, the improved variants of FW that we discuss will require that $\Coreset^{(t)}$ is maintained.

\vspace{-2mm}
\paragraph{Zig-Zagging Phenomenon.} %
When the optimal solution lies at the boundary of $\domain$, the convergence rate of
the iterates is slow, i.e. sublinear: $f(\x^{(t)}) - f(\x^*) \le O\big(1/t\big)$, for
$\x^*$ being an optimal
solution~\citep{Frank:1956vp,Canon:1968du,Dunn:1979da,Jaggi:2013wg}. %
This is because the iterates of the classical FW algorithm start to
zig-zag between the vertices defining the face containing the
solution $\x^*$ (see left of
Figure~\ref{fig:FWzigzag}). In fact, the $1/t$ rate is tight for a large class of functions: \citet{Canon:1968du,Wolfe:1970wy}
showed (roughly) that
$f(\x^{(t)}) - f(\x^*) \geq \Omega\big(1/t^{1+\delta}\big)$ for any $\delta
>0$ when $\x^*$ lies on a face of $\domain$ with some additional regularity
assumptions.
Note that this lower bound is different than the $\Omega\big(1/t\big)$ one
presented in \citep[Lemma 3]{Jaggi:2013wg} which holds for all one-atom-per-step algorithms but assumes high dimensionality $d \geq t$.

\begin{figure}[t]
  \vspace{-2.5mm}
  \centering
  \includegraphics[width = 0.32\textwidth]{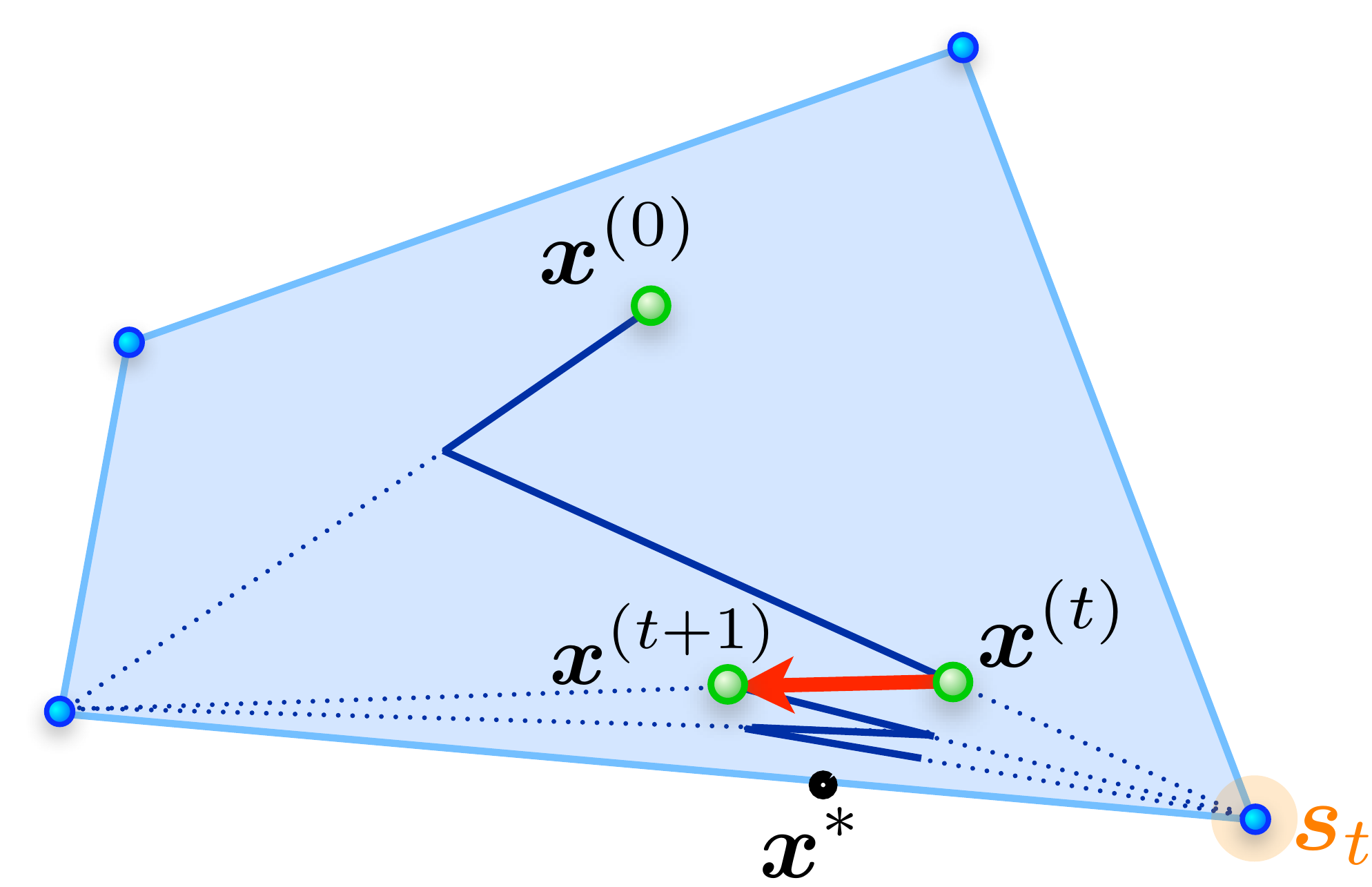} 
  \includegraphics[width = 0.32\textwidth]{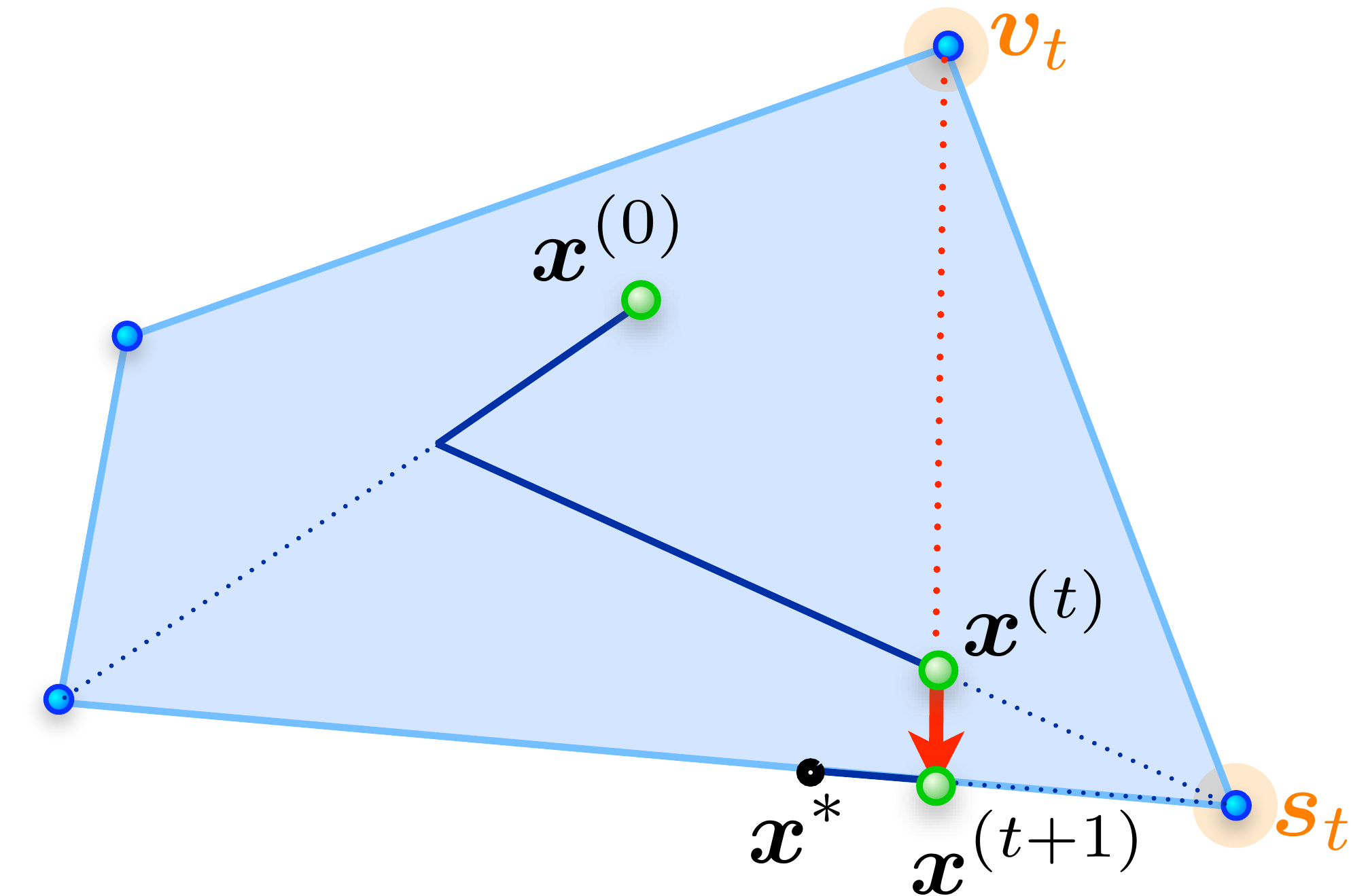} 
  \includegraphics[width = 0.32\textwidth]{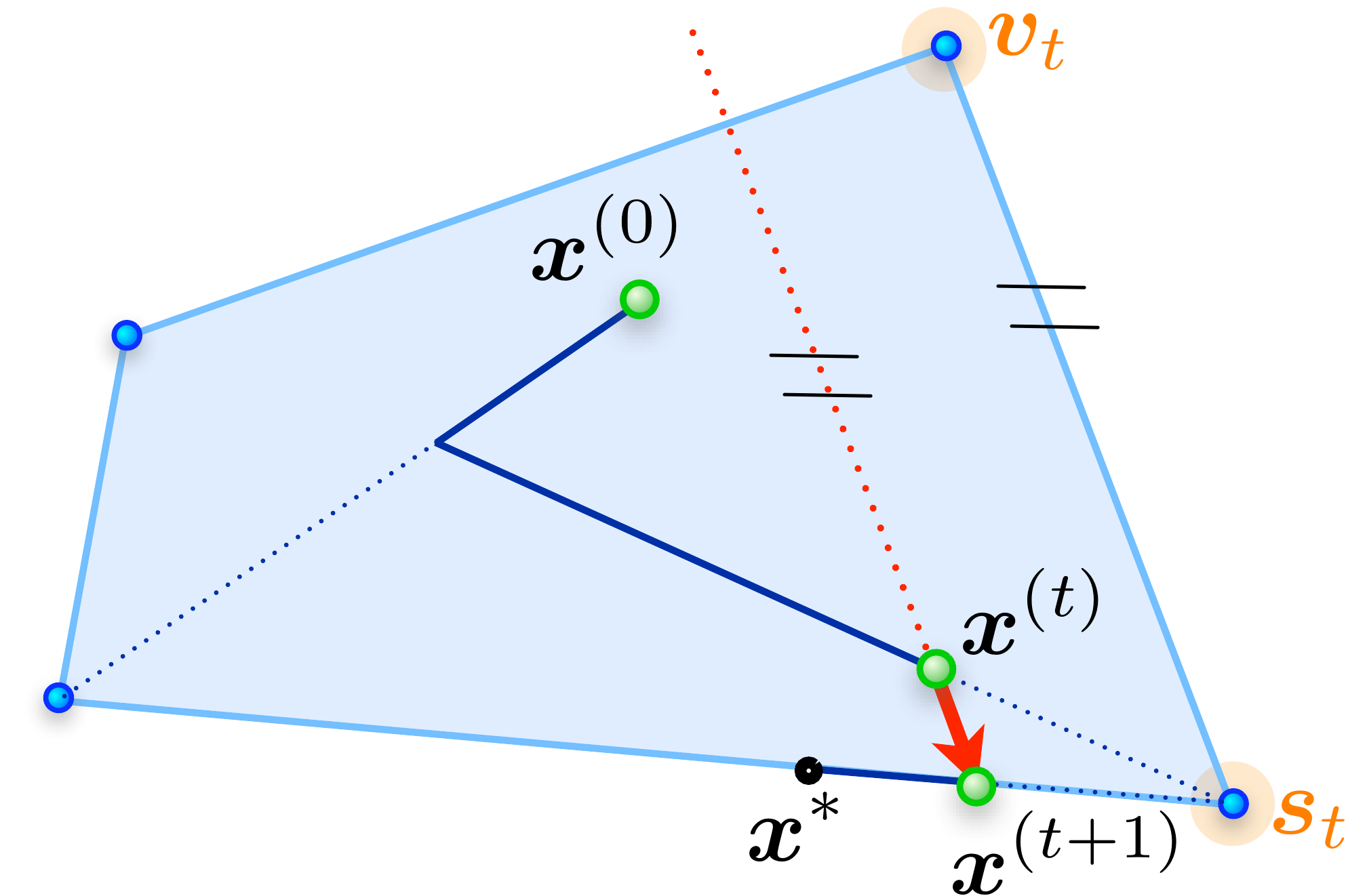} 
  \vspace{-2mm}
  \caption{\small (left) The FW algorithm zig-zags when the solution $\x^*$ lies on the boundary. (middle) Adding the possibility of an \emph{away step} attenuates this problem. (right) As an alternative, a pairwise FW step. 
  }
  \label{fig:FWzigzag} \vspace{1mm}
\end{figure}

\vspace{-2mm}
\section{Improved Variants of the Frank-Wolfe Algorithm}%
\label{sec:variants}\vspace{-2mm}

\setlength{\intextsep}{2mm}

\begin{algorithm}
	\caption{Away-steps Frank-Wolfe algorithm: \textbf{AFW}$(\x^{(0)}, \Vertices, \epsilon)$}
	\label{alg:AFW}
	\begin{algorithmic}[1]
	\STATE Let $\x^{(0)} \in \Vertices$, and $\Coreset^{(0)} := \{\x^{(0)}\}$
	   \qquad \emph{\small(so that $\alpha^{(0)}_{\vv} = 1$ for $\vv=\x^{(0)}$ and $0$ otherwise)}
	\FOR{$t=0\dots T$}
		\STATE Let $\s_t := \lmo \!\left(\nabla f(\x^{(t)})\right)$ %
		   and $\dd_t^\FW := \s_t - \x^{(t)}$ \qquad~~ \emph{\small(the FW direction)}
		\STATE Let $\vv_t \in \displaystyle\argmax_{\vv \in \Coreset^{(t)} } \textstyle\left\langle \nabla f(\x^{(t)}), \vv \right\rangle$ and $\dd_t^\away := \x^{(t)} - \vv_t$ \qquad \emph{\small(the away direction)}
		\STATE \textbf{if} $g_t^\FW  := \left\langle -\nabla f(\x^{(t)}), \dd_t^\FW\right\rangle  \leq \epsilon$ \textbf{then} \textbf{return} $\x^{(t)}$ \hspace{12mm}\emph{\small(FW gap is small enough, so return)}
		  \IF{$\left\langle -\nabla f(\x^{(t)}), \dd_t^\FW\right\rangle  \geq \left\langle -\nabla f(\x^{(t)}), \dd_t^\away\right\rangle$ }
		  \STATE $\dd_t :=  \dd_t^\FW$, and $\stepmax := 1$  
			     \hspace{22mm}\emph{\small(choose the FW direction)}
		  \ELSE
		  \STATE $\dd_t :=  \dd_t^\away$, and $\stepmax := \alpha_{\vv_t} / (1- \alpha_{\vv_t})$
		  	\hspace{5mm}\emph{\small(choose away direction; maximum feasible step-size)}
		  \ENDIF	
		  \STATE Line-search: $\stepsize_t \in \displaystyle\argmin_{\stepsize \in [0,\stepmax]} \textstyle f\left(\x^{(t)} + \stepsize \dd_t\right)$ 
		  \STATE Update $\x^{(t+1)} := \x^{(t)} + \stepsize_t \dd_t$  
		  	\hspace{2cm}\emph{\small(and accordingly for the weights $\aa^{(t+1)}$, see text)}
		  \STATE Update $\Coreset^{(t+1)} := \{\vv \in \Vertices \,\: \mathrm{ s.t. } \,\: \alpha^{(t+1)}_{\vv} > 0\}$
	\ENDFOR
	\end{algorithmic}
\end{algorithm}
\begin{algorithm}
	\caption{Pairwise Frank-Wolfe algorithm: \textbf{PFW}$(\x^{(0)}, \Vertices, \epsilon)$}
	\label{alg:PFW}
	\begin{algorithmic}[1]
	\STATE $\ldots$ as in Algorithm~\ref{alg:AFW}, except replacing lines~6 to~10 by:  $\,\, \dd_t = \dd^\PFW_t \!\!\!:= \s_t - \vv_t$, and $\stepmax := \alpha_{\vv_t}$.
	\end{algorithmic}
\end{algorithm}

\setlength{\intextsep}{\Sintextsep}

\vspace{-1.5mm}
\paragraph{Away-Steps Frank-Wolfe.}
To address the zig-zagging problem of FW, \citet{Wolfe:1970wy} proposed to
add the possibility to move \emph{away} from an active atom in
$\Coreset^{(t)}$ (see middle of Figure~\ref{fig:FWzigzag}); this simple
modification is sufficient to make the algorithm linearly convergent for
strongly convex functions. We describe the away-steps variant of Frank-Wolfe
in Algorithm~\ref{alg:AFW}.\footnote{%
The original algorithm presented in~\citep{Wolfe:1970wy} was not convergent;
this was corrected by~\citet{Guelat:1986fq}, %
assuming a tractable representation of~$\domain$ with linear inequalities and
called it the modified Frank-Wolfe (MFW) algorithm. Our description in
Algorithm~\ref{alg:AFW} extends it to the more general setup
of~\eqref{eq:optGenConvex}.
}
The \emph{away} direction $\dd_t^\away$ is defined in line~4 by finding the
atom $\vv_t$ in~$\Coreset^{(t)}$ that maximizes the potential of descent
given by $g_t^\away := \innerProd{-\nabla f(\x^{(t)})}{\x^{(t)} - \vv_t}$. 
Note that this search is over the (typically small) active set $\Coreset^{(t)}$, and is
fundamentally easier than the linear oracle $\lmo$. %
The maximum step-size $\stepmax$ as defined on line~9 ensures that the new
iterate $\x^{(t)} + \stepsize \dd_t^\away$ stays in $\domain$. %
In fact, this guarantees that the convex representation is maintained, and we 
stay inside $\conv(\Coreset^{(t)}) \subseteq \domain$.
When $\domain$ is a simplex, then the barycentric coordinates are unique and
$\x^{(t)} + \stepmax \dd_t^\away$ truly lies on the boundary of $\domain$. 
On the other hand, if $|\Vertices| > \dim(\domain)+1$ (e.g. for the cube), 
then it could hypothetically be possible to have a step-size bigger than $\stepmax$ which is still 
feasible. %
Computing the true
maximum feasible step-size would require the ability to know when we cross
the boundary of~$\domain$ along a specific line, which is not possible for
general~$\domain$. Using the conservative maximum step-size of line~9
ensures that we do not need this more powerful oracle. This is why
Algorithm~\ref{alg:AFW} requires to maintain $\Coreset^{(t)}$ (unlike
standard FW). %
Finally, as in classical FW, the FW gap $g_t^\FW$ is an upper bound on the unknown
suboptimality, and can be used as a stopping criterion: \vspace{-1mm}
\begin{equation*}
g_t^\FW := \innerProd{-\nabla f(\x^{(t)})}{\dd_t^\FW} \geq \innerProd{-\nabla
f(\x^{(t)})}{\x^* - \x^{(t)}} \geq f(\x^{(t)}) - f(\x^*) \quad \text{ (by
convexity)}.\vspace{-3mm}
\end{equation*}

If $\stepsize_t = \stepmax$, then we call this step a \emph{drop step}, as it
fully removes the atom $\vv_t$ from the currently active set of atoms
$\Coreset^{(t)}$ (by settings its weight to zero).
The weight updates for lines~12 and~13 are of the following form:
For a FW step, we have $\Coreset^{(t+1)} = \{\s_t\}$ if $\stepsize_t = 1$;
otherwise $\Coreset^{(t+1)} = \Coreset^{(t)} \cup \{\s_t\}$. Also, we have
$\alpha^{(t+1)}_{\s_t} := (1-\stepsize_t) \alpha^{(t)}_{\s_t} + \stepsize_t$
and $\alpha^{(t+1)}_{\vv} := (1-\stepsize_t) \alpha^{(t)}_{\vv}$ for $\vv \in
\Coreset^{(t)} \setminus  \{\s_t\}$. 
For an away step, we have $\Coreset^{(t+1)} = \Coreset^{(t)} \setminus
\{\vv_t\}$ if  $\stepsize_t = \stepmax$ (a \emph{drop step}); 
otherwise $\Coreset^{(t+1)} = \Coreset^{(t)}$.  Also, we have
$\alpha^{(t+1)}_{\vv_t} := (1+\stepsize_t) \alpha^{(t)}_{\vv_t} -
\stepsize_t$ and $\alpha^{(t+1)}_{\vv} := (1+\stepsize_t) \alpha^{(t)}_{\vv}$
for $\vv \in \Coreset^{(t)} \setminus  \{\vv_t\}$. \vspace{-2mm}

\paragraph{Pairwise Frank-Wolfe.}
The next variant that we present is inspired by an early algorithm
by~\citet{Mitchell:1974uy}, called the MDM %
algorithm, originally invented for the polytope distance problem. Here the
idea is to only move weight mass between two atoms in each step. More
precisely, the generalized method as presented in Algorithm~\ref{alg:PFW}
moves weight from the away atom~$\vv_t$ to the FW atom~$\s_t$, and keeps all
other $\alpha$ weights un-changed. 
We call such a swap of mass between the two atoms a \emph{pairwise FW} step,
i.e. $\alpha_{\vv_t}^{(t+1)} = \alpha_{\vv_t}^{(t)} - \stepsize$ and
$\alpha_{\s_t}^{(t+1)} = \alpha_{\s_t}^{(t)} + \stepsize$ for some step-size
$\stepsize \leq \stepmax := \alpha_{\vv_t}^{(t)}$. 
In contrast, classical FW shrinks all active weights at every iteration.

The pairwise FW direction will also be central to our proof technique to
provide the first global linear convergence rate for away-steps FW, as well
as the fully-corrective variant and Wolfe's min-norm-point algorithm. 

As we will see in Section~\ref{sec:theorems}, the rate guarantee for the
pairwise FW variant is more loose than for the other variants, because we cannot provide a satisfactory bound on the number of the problematic \emph{swap steps} (defined just before Theorem~\ref{thm:megaConvergenceTheorem}).
Nevertheless, the algorithm seems to perform quite well in practice, often
outperforming away-steps FW, especially in the important case of
sparse solutions, that is if the optimal solution
$\x^*$ lies on a low-dimensional face of $\domain$ (and thus one wants to
keep the active set $\Coreset^{(t)}$ small). The pairwise FW step is
arguably more efficient at pruning the coordinates in $\Coreset^{(t)}\!$. In
contrast to the away step which moves the mass back \emph{uniformly} onto
all other active elements $\Coreset^{(t)}$ (and might require more
corrections later), the pairwise FW step only moves the mass onto the (good) FW
atom $\s_t$.
A slightly different version than Algorithm~\ref{alg:PFW} was also proposed
by~\citet{Nanculef:2014bj}, though their convergence proofs were incomplete
(see~Appendix~\ref{sec:PFWdetails}).
The algorithm is related to classical working set algorithms, such as
the SMO algorithm used to train SVMs~\citep{Platt:1999wl}. We refer
to~\citep{Nanculef:2014bj} for an empirical comparison for SVMs, as
well as their Section~5 for more related work. See also Appendix~\ref{sec:PFWdetails} 
for a link between pairwise FW and~\citep{Garber:2013vl}.
\vspace{-2mm}

\paragraph{Fully-Corrective Frank-Wolfe, and Wolfe's Min-Norm Point
Algorithm.}

When the linear oracle is expensive, it might be worthwhile to do more work
to optimize over the active set $\Coreset^{(t)}$ in between each call to the
linear oracle, rather than just performing an away or pairwise step. We give
in Algorithm~\ref{alg:FCFW} the fully-corrective Frank-Wolfe (FCFW) variant,
that maintains a correction polytope defined by a set of atoms
$\Vertices^{(t)}$ (potentially larger than the active set $\Coreset^{(t)}$).
Rather than obtaining the next iterate by line-search, $\x^{(t+1)}$ is
obtained by re-optimizing $f$ over $\conv(\Vertices^{(t)})$. Depending on how
the correction is implemented, and how the correction atoms $\Vertices^{(t)}$
are maintained, several variants can be obtained. 
These variants are known
under many names, such as the extended FW method
by~\citet{Holloway:1974:FCFW} or the simplicial decomposition
method~\citep{Hohenbalken:1977:simplicial,Hearn:1987:simplicial}. 
Wolfe's min-norm point (MNP) algorithm~\citep{Wolfe:1976:MNP} for polytope distance problems is often confused with 
FCFW for quadratic objectives.  The major difference is that standard FCFW optimizes $f$ over $\conv(\Vertices^{(t)})$, whereas MNP implements the correction as a sequence of affine projections that potentially yield a different update, but can be computed more efficiently in several practical applications~\citep{Wolfe:1976:MNP}.
We describe precisely in Appendix~\ref{sec:MNPdetails} a generalization of the MNP algorithm as a specific case of the correction subroutine from step~7 of the generic Algorithm~\ref{alg:FCFW}.\vspace{-1mm}

\begin{algorithm}
	\caption{Fully-corrective Frank-Wolfe with approximate correction: \textbf{FCFW}$(\x^{(0)}, \Vertices, \epsilon)$}
	\label{alg:FCFW}
	\begin{algorithmic}[1]
	\STATE \textbf{Input:} Set of atoms $\Vertices$, active set $\Coreset^{(0)}$, starting point $\x^{(0)}= \displaystyle \sum_{\vv \in \Coreset^{(0)}} \alpha^{(0)}_{\vv} \vv$, stopping criterion $\epsilon$.
	\STATE Let $\Vertices^{(0)} := \Coreset^{(0)}$ \quad (optionally, a bigger $\Vertices^{(0)}$ could be passed as argument for a warm start)
	\FOR{$t=0\dots T$}
		\STATE Let $\s_t := \lmo \!\left(\nabla f(\x^{(t)})\right)$%
			\hspace{4cm} \emph{\small(the FW atom)}
		\STATE Let $\dd_t^\FW := \s_t - \x^{(t)}$ and $g_t^\FW = \innerProd{-\nabla f(\x^{(t)})}{\dd_t^\FW}$ \qquad \emph{\small(FW gap)}
		\STATE \textbf{if} $g_t^\FW \leq \epsilon$ \textbf{then} \textbf{return} $\x^{(t)}$
			\STATE $(\x^{(t+1)}, \Vertices^{(t+1)}) := \textbf{\textrm{Correction}}(\x^{(t)}, \Vertices^{(t)}, \s_t, \epsilon)$ \hspace{1.2cm} \emph{\small (approximate correction step)}
	\ENDFOR
	\end{algorithmic}
\end{algorithm}

\begin{algorithm}
	\caption{Approximate correction: \textbf{Correction}$(\x^{(t)}, \Vertices^{(t)}, \s_t, \epsilon)$}
	\label{alg:correction}
	\begin{algorithmic}[1]
	\STATE Return $(\x^{(t+1)},  \Vertices^{(t+1)})$ with the following properties:
	\STATE \quad $\Coreset^{(t+1)}$ is the active set for $\x^{(t+1)}$ and $ \Vertices^{(t+1)} \supseteq \Coreset^{(t+1)}$.
	\STATE \quad $f(\x^{(t+1)}) \leq \displaystyle\min_{\stepsize \in [0,1]} \textstyle f\left(\x^{(t)} + \stepsize (\s_t - \x^{(t)})\right)$ \qquad \emph{\small (make at least as much progress as a FW step)}
	\STATE \quad $g_{t+1}^\away := \displaystyle\max_{\vv \in \Coreset^{(t+1)} } \textstyle\left\langle -\nabla f(\x^{(t+1)}), \x^{(t+1)} - \vv \right\rangle \leq \epsilon$ \quad \emph{\small (the away gap is small enough)}
	\end{algorithmic}
\end{algorithm}

The original convergence analysis of the FCFW
algorithm~\citep{Holloway:1974:FCFW}  (and also MNP
algorithm~\citep{Wolfe:1976:MNP}) only showed that they were finitely
convergent, with a bound on the number of iterations in terms of the
cardinality of $\Vertices$ (unfortunately an exponential %
number in general). \citet{Holloway:1974:FCFW} also argued that FCFW had an
asymptotic linear convergence based on the flawed argument
of~\citet{Wolfe:1970wy}. As far as we know, our work is the first to provide
global linear convergence rates for FCFW and MNP for general strongly convex
functions. Moreover, the proof of convergence for FCFW does not require
an exact solution to the correction step; instead, we show that the weaker 
properties stated for the approximate correction procedure in Algorithm~\ref{alg:correction} 
are sufficient 
for a global linear convergence rate (this correction could be implemented
using away-steps FW, as done for example in~\cite{Krishnan:2015ws}).
\vspace{-2mm}

\section{Global Linear Convergence Analysis}\vspace{-2mm}

\subsection{Intuition for the Convergence Proofs}\vspace{-2mm}

We first give the general intuition for the linear convergence proof of the
different FW variants, starting from the work of~\citet{Guelat:1986fq}.
We assume that the objective function $f$ is smooth over a compact set $\domain$, i.e. its gradient is Lipschitz continuous with constant $L$. Also let $M := \diam(\domain)$. Let
$\dd_t$ be the direction in which the line-search is executed by the
algorithm (Line~11 in Algorithm~\ref{alg:AFW}). By the standard descent
lemma~\citep[see e.g. (1.2.5) in][]{Nesterov:2004:lectures}, we have:\vspace{-0.5mm}
\begin{equation} \label{eq:descentLemma}
f(\x^{(t+1)}) \leq f(\x^{(t)}+\stepsize \dd_t) \leq f(\x^{(t)}) + \stepsize
\innerProd{\nabla f(\x^{(t)})}{\dd_t} + \frac{\stepsize^2}{2} L \| \dd_t \|^2
\quad \forall \stepsize \in [0, \stepmax].
\end{equation}
We let $\r_t := - \nabla f(\x^{(t)})$ and let $h_t := f(\x^{(t)}) - f(\x^*)$
be the suboptimality error. Supposing for now that $\stepmax \geq
\stepsize_t^* :=  \innerProd{\r_t}{\dd_t} / (L \|\dd_t\|^2$). We can set
$\stepsize = \stepsize_t^*$ to minimize the RHS of~\eqref{eq:descentLemma},
subtract $f(\x^*)$ on both sides, and re-organize to get a lower bound on the
progress:\vspace{-0.5mm}
\begin{equation} \label{eq:lowerProgress}
h_t- h_{t+1} \geq \frac{\innerProdCompressed{\r_t}{\dd_t}^2}{2 L
\|\dd_t\|^2} 
= \frac{1}{2L} \innerProdCompressed{\r_t}{\hat{\dd}_t}^2 \, , \vspace{-0.5mm}
\end{equation}
where we use the `hat' notation to denote normalized vectors: $\hat{\dd}_t := \dd_t / \|\dd_t \|$. 
Let $\err_t := \x^* - \x^{(t)}$ be the error vector. By $\mu$-strong
convexity of $f$, we have:\vspace{-1mm}
\begin{equation} \label{eq:strongLower}
f(\x^{(t)} + \stepsize \err_t) \geq f(\x^{(t)}) + \stepsize \innerProd{\nabla
f(\x^{(t)})}{\err_t} + \frac{\stepsize^2}{2} \mu \| \err_t \|^2 \quad \forall
\stepsize \in [0,1].\vspace{-0.5mm}
\end{equation}
The RHS is lower bounded by its minimum as a function of $\stepsize$
(unconstrained), achieved using $\stepsize := \innerProdCompressed{\r_t}{\err_t} / (\mu \| \err_t \|^2)$.
We are then free to use any
value of $\stepsize$ on the LHS and maintain a valid bound. In particular, we
use $\stepsize = 1$ to obtain $f(\x^*)$. Again re-arranging, we get:\vspace{-0.5mm}
\begin{equation} \label{eq:linearProgressBound}
h_t \leq \frac{\innerProd{\r_t}{\hat{\err}_t}^2}{2 \mu} \, , \quad \text{ and
combining with \eqref{eq:lowerProgress}, we obtain: } \quad h_{t} - h_{t+1}
\geq \frac{\mu}{L}
\frac{\innerProdCompressed{\r_t}{\hat{\dd}_t}^2}{\innerProdCompressed{\r_t}{\hat{\err}_t}^2} \, h_t .
\end{equation}
The inequality~\eqref{eq:linearProgressBound} is
fairly general and valid for any line-search method in direction $\dd_t$. To
get a linear convergence rate, we need to lower bound (by a positive
constant) the term in front of $h_t$ on the RHS, which depends on the angle
between the update direction $\dd_t$ and the negative gradient~$\r_t$. If we
assume that the solution $\x^*$ lies in the relative interior of $\domain$
with a distance of at least $\delta > 0$ from the boundary, then
$\innerProdCompressed{\r_t}{\dd_t} \geq \delta \|\r_t\|$ for the FW direction
$\dd_t^\FW$, and by combining with $\|\dd_t\| \leq M$, we get a linear rate with constant
$1-\frac{\mu}{L}(\frac{\delta}{M})^2$ (this was the result
from~\citep{Guelat:1986fq}). On the other hand, if~$\x^*$ lies
on the boundary, then $\innerProdCompressed{\hat{\r}_t}{\hat{\dd}_t}$ %
gets
arbitrary close to zero for standard FW (the zig-zagging phenomenon) and the
convergence is sublinear.

\paragraph{Proof Sketch for AFW.} The key insight to prove the global linear
convergence for AFW is to relate $\innerProdCompressed{\r_t}{\dd_t}$ with the
\emph{pairwise FW} direction $\dd_t^\PFW := \s_t - \vv_t$. By the way the
direction $\dd_t$ is chosen on lines~6 to~10 of Algorithm~\ref{alg:AFW}, we have:
\begin{equation} \label{eq:AFWgapInequality}
2 \innerProdCompressed{\r_t}{\dd_t} \geq
\innerProdCompressed{\r_t}{\dd_t^\FW} + \innerProdCompressed{\r_t}{\dd_t^A} =
\innerProdCompressed{\r_t}{\dd_t^\FW+\dd_t^A} = 
\innerProdCompressed{\r_t}{\dd_t^\PFW} .  
\end{equation}
We thus have $\innerProdCompressed{\r_t}{\dd_t} \geq
\innerProdCompressed{\r_t}{\dd_t^\PFW} / 2$. Now the crucial property of the
pairwise FW direction is that for any potential negative gradient direction
$\r_t$, the worst case inner product $\innerProdCompressed{\hat{\r}_t}{\dd_t^\PFW}$
\begin{wrapfigure}{r}{3.5cm}\vspace{-3mm}\hspace{-3mm}
\includegraphics[width=1.06\linewidth]{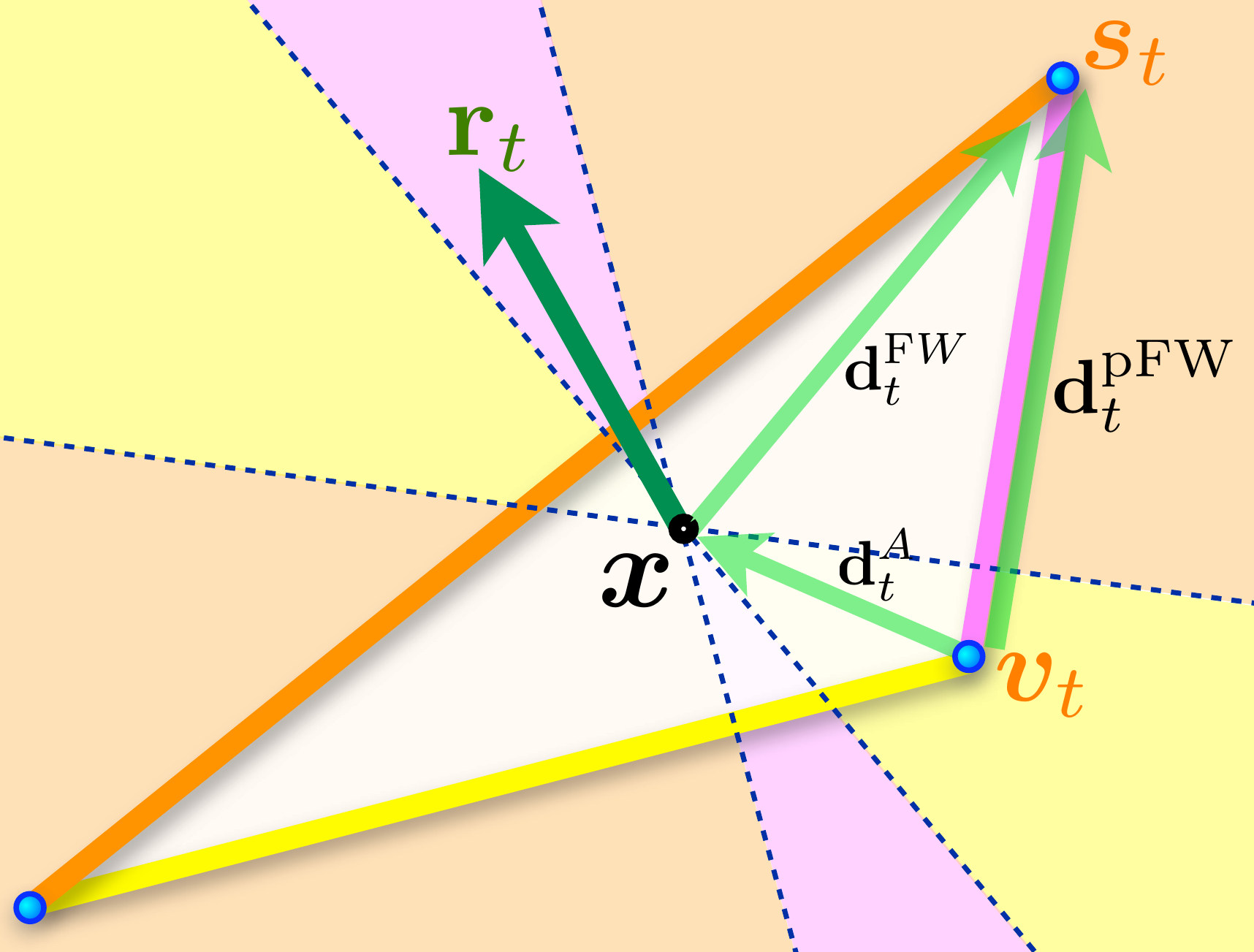}\vspace{-0.5em}
\label{fig:pwidth}
\vspace{-1.5em}
\end{wrapfigure}
can be lower bounded away from
zero by a quantity depending only on the geometry of $\domain$ (unless we are at the optimum).
We call this quantity the \emph{pyramidal width} of $\Vertices$. 
The figure on the right shows the six possible pairwise FW directions $\dd_t^\PFW$ for a triangle domain, depending on which colored area the $\r_t$ direction falls into. We will see that the pyramidal width is
related to the smallest width of pyramids that we can construct from
$\Vertices$ in a specific way related to the choice of the away and towards atoms $\vv_t$ and $\s_t$. See~\eqref{eq:Pwidth} and our main Theorem~\ref{thm:muFdirWinterpretation} in Section~\ref{sec:Pwidth}.

This gives the main argument for the linear convergence of AFW for steps
where $\gamma_t^* \leq \stepmax$. When~$\stepmax$ is too small, AFW will perform a
\emph{drop step}, as the line-search will truncate the step-size to
$\stepsize_t = \stepmax$. We cannot guarantee sufficient progress in this
case, but the drop step decreases the active set size by one, and thus they
cannot happen too often (not more than half the time). These are the main
elements for the global linear convergence proof for AFW. The rest is to
carefully consider various boundary cases.
We can re-use the same 
techniques to prove the convergence for pairwise FW, though unfortunately the
latter also has the possibility of problematic \emph{swap steps}. While their
number can be bounded, so far we only found the extremely loose bound quoted in Theorem~\ref{thm:megaConvergenceTheorem}.\vspace{-2mm}

\paragraph{Proof Sketch for FCFW.} For FCFW, by line~4 of the
correction Algorithm~\ref{alg:correction}, 
the away gap satisfies $g_t^\away \leq \epsilon$
at the beginning of a new iteration.
Supposing that the algorithm does not exit at line~6 of
Algorithm~\ref{alg:FCFW}, we have $g_t^\FW > \epsilon$ and therefore $2
\innerProdCompressed{\r_t}{\dd_t^\FW} \geq
\innerProdCompressed{\r_t}{\dd_t^\PFW}$ using a similar argument as
in~\eqref{eq:AFWgapInequality}. Finally, by line~3 of Algorithm~\ref{alg:correction}, the correction is
guaranteed to make at least as much progress as a line-search in direction
$\dd_t^\FW$, and so the progress bound~\eqref{eq:linearProgressBound} applies also to FCFW.\vspace{-2mm}

\subsection{Convergence Results}\label{sec:theorems}\vspace{-2mm}
We now give the global linear convergence rates for the four variants of the
FW algorithm: away-steps FW (AFW Alg.~\ref{alg:AFW}); pairwise FW
(PFW Alg.~\ref{alg:PFW}); fully-corrective FW (FCFW
Alg.~\ref{alg:FCFW} with approximate correction Alg.~\ref{alg:correction}); and Wolfe's
min-norm point algorithm (Alg.~\ref{alg:FCFW} with
MNP-correction as Alg.~\ref{alg:MNP} in Appendix~\ref{sec:MNPdetails}). For the AFW, MNP  and PFW algorithms, we call
a \emph{drop step} when the active set shrinks $|S^{(t+1)}| < |S^{(t)}|$. For
the PFW algorithm, we also have the possibility of a \emph{swap step} where
$\stepsize_t = \stepmax$ but $|S^{(t+1)}| = |S^{(t)}|$ (i.e. the mass was
fully swapped from the away atom to the FW atom). 
A nice property of FCFW is that it does not have any drop step (it executes both FW steps and away steps simultaneously while guaranteeing enough progress at
every iteration).

\begin{theorem}\label{thm:megaConvergenceTheorem}
Suppose that $f$ has $L$-Lipschitz gradient\footnote{For AFW and PFW, we
actually require that $\nabla f$ is $L$-Lipschitz over the larger
domain $\domain + \domain - \domain$.} and is $\mu$-strongly convex over
$\domain=\conv(\Vertices)$. Let $M=\diam(\domain)$ and $\delta =
\PWidth(\Vertices)$ as defined by~\eqref{eq:Pwidth}.
Then the suboptimality~$h_t$ of the iterates of all the four variants of the FW algorithm %
decreases geometrically at each step that is not a drop step nor a swap step
(i.e. when $\stepsize_t < \stepmax$, called a `good step'), that is\vspace{-1mm}
\[
h_{t+1} \leq \left(1-\rho\right) h_t \ , \quad\quad \text{ where } \rho :=
\frac{\mu}{4L}\left(\frac{\delta}{M}\right)^2. \vspace{-3mm}
\]

Let $k(t)$ be the number of `good steps' up to iteration $t$.
We have $k(t) = t$ for FCFW; $k(t) \geq t/2$ for
MNP and AFW; and $k(t) \geq t/(3|\Vertices|!+1)$ for PFW (because of the swap steps).
This yields a global linear convergence rate of $h_t \leq h_0 \exp(-
\rho \, k(t))$ for all variants. If $\mu=0$ (general convex),
then $h_t = O(1/k(t))$ instead. See Theorem~\ref{thm:megaConvergenceTheorem2} in Appendix~\ref{sec:conv} for an affine invariant version and proof.
\vspace{-1mm}
\end{theorem}

Note that to our knowledge, none of the existing linear convergence results showed that the duality gap was also linearly convergent. The result for the gap follows directly from the simple manipulation of~\eqref{eq:descentLemma}; putting the FW gap to the LHS
and optimizing the RHS for $\stepsize \in [0,1]$.

\begin{theorem}\label{thm:gapTheorem}
Suppose that $f$ has $L$-Lipschitz gradient over $\domain$ with $M :=
\diam(\domain)$. Then the FW gap $g_t^\FW$ for \emph{any} algorithm is upper
bounded by the primal error $h_t$ as follows:
\begin{equation}
g_t^\FW \leq h_t + LM^2/2 \,\textnormal{ when }\, h_t > LM^2/2, 
\quad\quad\quad\quad  g_t^\FW \leq M \sqrt{2 h_t L} \,\,\textnormal{
otherwise }.
\end{equation}
\end{theorem}

\section{Pyramidal Width} \label{sec:Pwidth}\vspace{-3mm}

We now describe the claimed lower bound on the angle between
the negative gradient and the pairwise FW direction, which depends only on
the geometric properties of~$\domain$. According to our argument about the
progress bound~\eqref{eq:linearProgressBound} and the PFW
gap~\eqref{eq:AFWgapInequality}, our goal is to find a lower bound on
$\innerProdCompressed{\r_t}{\dd_t^\PFW}/\innerProdCompressed{\r_t}{\hat{\err}
_t}$. First note that $\innerProdCompressed{\r_t}{\dd_t^\PFW}\! = \!\innerProdCompressed{\r_t}{\s_t\! -\! \vv_t} \! = \hspace{-5mm}\displaystyle
\max_{\s \in \domain, \vv \in \Coreset^{(t)}}\hspace{-3mm} \innerProdCompressed{\r_t}{\s -
\vv}$ where $\Coreset^{(t)}$ is a possible active set for~$\x^{(t)}$. This
looks like the \emph{directional width} of a pyramid with base~$\Coreset^{(t)}$ 
and summit $\s_t$. To be conservative, we consider the worst
case possible active set for~$\x^{(t)}$; this is what we will call the
\emph{pyramid directional width}~$\PdirW(\Vertices, \r_t, \x^{(t)})$. We 
start with the following definitions.

\vspace{-3mm}
\paragraph{Directional Width.}
The directional width of a set $\Vertices$ with respect to a direction $\r$
is defined as $\dirW(\Vertices,\r) := \max_{\s, \vv \in\Vertices}
\big\langle \frac{\r}{\norm{\r}}, \s - \vv \big\rangle$. The \emph{width} of~$\Vertices$
is the minimum directional width over all possible directions in its affine
hull.

\vspace{-3mm}
\paragraph{Pyramidal Directional Width.} We define the pyramidal directional
width of a set $\Vertices$ with respect to a direction $\r$ and a base point
$\x \in \domain$ to be\vspace{-0.5mm}
\begin{equation} \label{eq:TruePdirW}
\PdirW(\Vertices,\r, \x) := \min_{\S \in \S_{\x}} \dirW( \S \cup
\{\s(\Vertices,\r) \} , \; \r) = \min_{\S \in \S_{\x}} \max_{\s \in
\Vertices, \vv \in \S} \textstyle \big\langle \frac{\r}{\norm{\r}}, \s - \vv \big\rangle
,
\end{equation}
where $\S_{\x} := \{ \S \, | \, \S \subseteq \Vertices$ such that $\x$ is a
proper\footnote{By \emph{proper} convex combination, we mean that all
coefficients are non-zero in the convex combination.} convex combination of
all the elements in $\S\}$, and $\s(\Vertices,\r) := \argmax_{\vv \in
\Vertices} \innerProdCompressed{\r}{\vv}$ is the FW atom used as a summit. 

\vspace{-2mm}
\paragraph{Pyramidal Width.}
To define the pyramidal width of a set, we take the minimum over the cone of
possible \emph{feasible} directions~$\r$ (in order to avoid the problem of zero width).\\
A direction~$\r$ is \emph{feasible} for $\Vertices$ from $\x$ if it points inwards $\conv(\Vertices)$, 
(i.e. $\r \in \text{cone}(\Vertices-\x)$).\\
We define the \emph{pyramidal width} of a set $\Vertices$ to be the smallest pyramidal width of all its faces, i.e.\vspace{-0.5mm}
\begin{equation} \label{eq:Pwidth}
\PWidth(\Vertices) := \displaystyle \min_{\substack{\Kface \in \textrm{faces}(\conv(\Vertices)) \\
												  \x \in \Kface \\
												  \r \in \text{cone}(\Kface-\x) \setminus \{\0\}} 
                                   } \PdirW(\Kface \cap \Vertices,\r, \x)                   .    \vspace{-1mm}                             
\end{equation}

\begin{theorem} \label{thm:muFdirWinterpretation}
Let $\x \in \domain=\conv(\Vertices)$ be a suboptimal point and $\S$ be an
active set for $\x$. Let~$\x^*$ be an optimal point and corresponding error
direction $\hat{\err} = (\x^*-\x)/\norm{\x^*-\x}$, and negative gradient $\r
:= -\nabla f(\x)$  (and so $\innerProdCompressed{\r}{\hat{\err}} > 0$). Let
$\dd = \s \!- \!\vv$ be the pairwise FW direction obtained over $\Vertices$ and $\S$ with negative
gradient $\r$. Then \vspace{-2mm}
\begin{equation} \label{eq:PWidthIsBound}
\frac{\innerProdCompressed{\r}{\dd}}{\innerProdCompressed{\r}{\hat{\err}}}
\geq \PWidth(\Vertices) .\vspace{-3mm}
\end{equation}   
\end{theorem}

\subsection{Properties of Pyramidal Width and Consequences}\vspace{-2mm}

\paragraph{Examples of Values.} The pyramidal width of a set $\Vertices$ is
lower bounded by the minimal width over all subsets of atoms, and thus is
strictly greater than zero if the number of atoms is finite.
On the other hand, this lower bound is often too loose to be useful, as in
particular, vertex subsets of the unit cube in dimension $d$ can have exponentially small width $O(d^{-\frac{d}{2}})$~\citep[see
Corollary 27 in][]{Ziegler:1999:01polytope}.
On the other hand, as we show here, the pyramidal width of the unit cube is actually
$1/\sqrt{d}$, justifying why we kept the tighter but more involved
definition~\eqref{eq:Pwidth}. See Appendix~\ref{app:cubeWidth} for the proof.

\begin{lemma}\label{lem:cubeWidth}
The pyramidal width of the unit cube in $\R^d$ is $1/\sqrt{d}$.\vspace{-2mm}
\end{lemma}
For the probability simplex with $d$ vertices, the pyramidal width is
actually the same as its width, which is $2/\sqrt{d}$ when $d$ is
even, and $2/\sqrt{d\!-\!1/d}$ when $d$ is odd~\citep{Alexander:1977:simplex} (see Appendix~\ref{app:cubeWidth}). 
In contrast, the pyramidal width of an infinite set can be zero. For
example, for a curved domain, the set of active atoms $\S$ can contain
vertices %
forming a very narrow pyramid, yielding a zero width in the limit. \vspace{-2mm}

\paragraph{Condition Number of a Set.} The inverse of the rate constant
$\rho$ appearing in Theorem~\ref{thm:megaConvergenceTheorem} is the product
of two terms: $L/\mu$ is the standard \emph{condition number} of the objective function appearing in
the rates of gradient methods in convex optimization. The second quantity
$(M/\delta)^2$ (diameter over pyramidal width) can be interpreted as
a \emph{condition number} of the domain $\domain$, or its
\emph{eccentricity}. %
The more eccentric the constraint set (large diameter
compared to its pyramidal width), the slower the convergence. 
The best condition number of a function is when its level sets are spherical; the
analog in term of  the constraint sets is actually the regular simplex, which has the
maximum width-to-diameter ratio amongst all simplices~\citep[see Corollary 1
in][]{Alexander:1977:simplex}. Its eccentricity is (at most) $d/2$. In contrast, the
eccentricity of the unit cube is  $d^2$, which is much worse.

We conjecture that the pyramidal width of a set of \emph{vertices} (i.e. extrema of their convex hull) is \emph{non-increasing} when another
vertex is added (assuming that all previous points remain vertices). 
For example, the unit cube can be obtained by iteratively adding vertices to the
regular probability simplex, and the pyramidal width thereby decreases from
$2/\sqrt{d}$ to $1/\sqrt{d}$.
This property could provide lower bounds for the pyramidal width of more complicated polytopes, such 
as $1/\sqrt{d}$ for the $d$-dimensional marginal polytope, 
as it can be obtained by removing vertices from the unit cube. 
\vspace{-3mm}
\paragraph{Complexity Lower Bounds.}  %
Combining the convergence Theorem~\ref{thm:megaConvergenceTheorem}
and the condition number of the unit simplex, we get a complexity of $O(d
\frac{L}{\mu} \log(\frac{1}{\epsilon}))$ to reach $\epsilon$-accuracy 
when optimizing a strongly convex function over the unit simplex. Here the
linear dependence on $d$ should not come as a surprise, 
in view of the known lower bound of $1/t$ for $t \leq d$ for Frank-Wolfe type
methods~\citep{Jaggi:2013wg}.

\vspace{-3mm}
\paragraph{Applications to Submodular Minimization.}
See~Appendix~\ref{app:submodular} for a consequence of our linear rate for the popular MNP algorithm for submodular function optimization (over the base polytope).

\vspace{-4mm}
\section{Non-Strongly Convex Generalization} \label{sec:nonStronglyConvex}
\vspace{-2mm}
Building on the work of~\citet{Beck:2015vo}
and~\citet{Wang:2014:nonStronglyConvex}, we can generalize our global linear
convergence results for all Frank-Wolfe variants for the more general case
where $f(\x) := g(\A \x) + \innerProd{\bv}{\x}$, for $\A \in \R^{\ydim \times
\xdim}$, $\bv \in \R^{\xdim}$ and where $g$ is $\mu_g$-strongly convex and
continuously differentiable over $\A \domain$. We note that for a general
matrix $\A$, $f$ is convex but not necessarily \emph{strongly} convex. In
this case, the linear convergence still holds but with the constant $\mu$
appearing in the rate of Theorem~\ref{thm:megaConvergenceTheorem} replaced
with the generalized constant $\tilde{\mu}$ appearing in
Lemma~\ref{lem:generalizedStrongConvexity} in
Appendix~\ref{app:NonStronglyConvex}.

\begin{wrapfigure}{r}{4.8cm}\vspace{-1.5em}
  \vspace{-8mm}
  \begin{center}
  \includegraphics[width=0.95\linewidth]{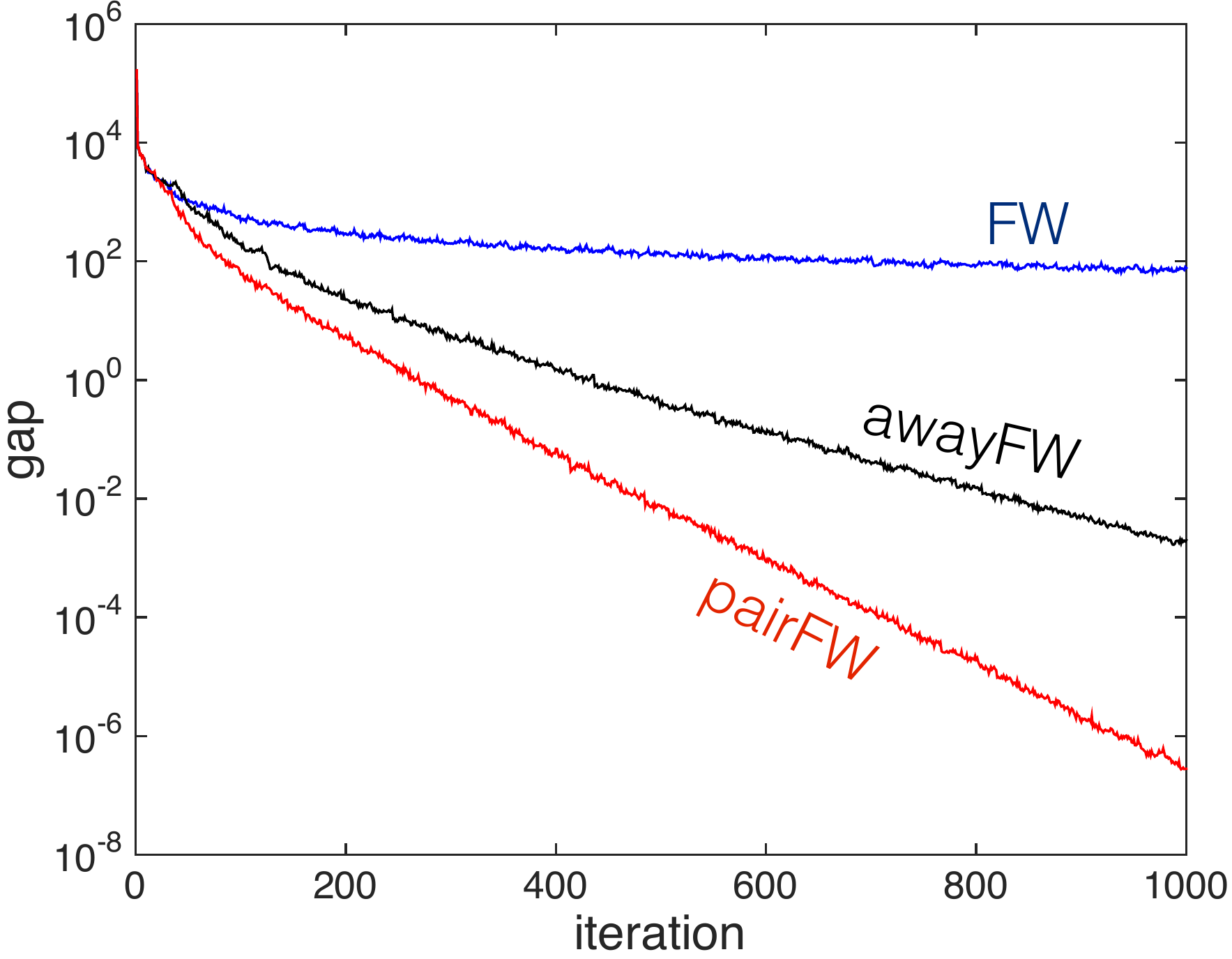}
  \\
  \ \includegraphics[width=0.95\linewidth]{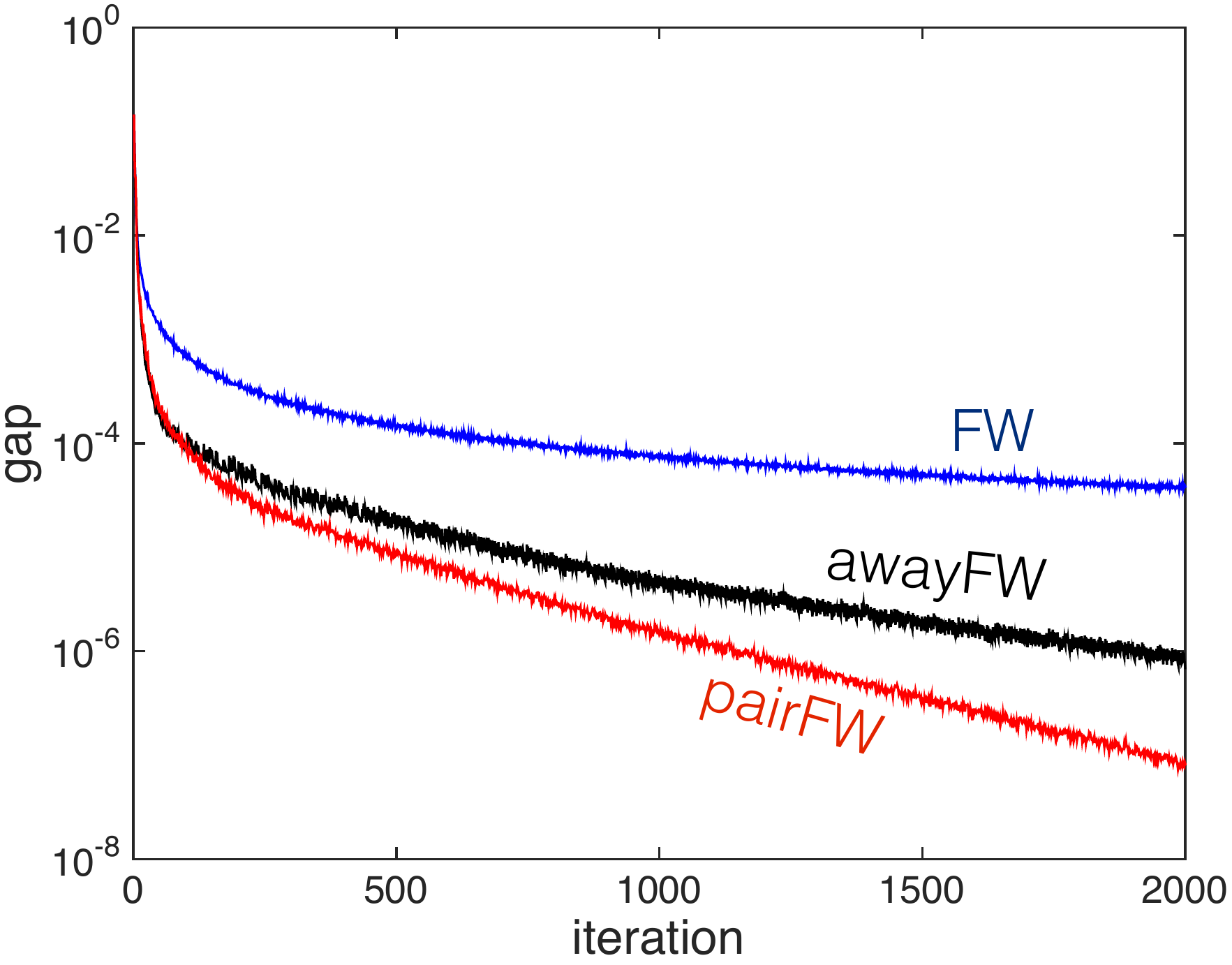}
  \vspace{-1em}
  \caption{\small Duality gap $g_t^\FW$ vs iterations on the Lasso problem (top), and video co-localization (bottom). 
  Code is available from the authors' website.
  }
  \label{fig:experiments}
  \end{center}  \vspace{-1em}
\end{wrapfigure}

\vspace{-4mm}
\section{Illustrative Experiments}
\vspace{-2mm}

We illustrate the performance of the presented algorithm variants in two
numerical experiments, shown in Figure~\ref{fig:experiments}.
The first example is a constrained Lasso problem ($\ell_1$-regularized least
squares regression), that is $\min_{\x \in \domain} \, f(\x) =
\norm{\A\x-\bv}^2$, with $\domain=20\cdot L_1$ a scaled $L_1$-ball. We used a
random Gaussian matrix $\A\in \R^{200 \times 500}$, and a noisy measurement
$\bv=\A\x^*$ with $\x^*$ being a sparse vector with 50 entries $\pm1$, and
$10\%$ of additive noise. For the $L_1$-ball, the linear minimization oracle
$\lmo$ just selects the column of $\A$ of best inner product with the residual vector.
The second application comes from video co-localization. The approach used
by~\citep{Joulin:2014uw} is formulated as a quadratic program (QP) over a
flow polytope, the convex hull of paths in a network. In this application,
the linear minimization oracle is equivalent to finding a shortest path in
the network, which can be done easily by dynamic
programming. %
For the $\lmo$, we re-use the code provided
by~\citep{Joulin:2014uw} %
and their included \textsf{\small aeroplane} dataset resulting in a QP over
660 variables.
In both experiments, we see that the modified FW variants (away-steps and
pairwise) outperform the original FW algorithm, and exhibit a linear convergence.
In addition, the constant in the convergence rate of Theorem~\ref{thm:megaConvergenceTheorem}
can also be empirically shown to be fairly tight for AFW and PFW by running them
on an increasingly obtuse triangle (see Appendix~\ref{app:triangle}).

\vspace{-2mm}
\paragraph{Discussion.}
Building on a preliminary version of our work~\citep{lacoste2013affine},
\citet{Beck:2015vo} also proved a linear rate for away-steps FW, but
with a simpler lower bound for the LHS of~\eqref{eq:PWidthIsBound} using
linear duality arguments. However, their lower bound
\citep[see e.g. Lemma 3.1 in][]{Beck:2015vo} is looser: they
get a $d^2$ constant for the eccentricity of the regular simplex instead
of the tighter $d$ that we proved. 
Finally, the recently proposed 
generic scheme for \emph{accelerating} first-order optimization 
methods in the sense of Nesterov from~\cite{lin2015catalyst}
applies directly to the FW variants given their global 
linear convergence rate that we proved.
This gives for the first time
first-order methods that \emph{only use linear oracles} 
and obtain the ``near-optimal'' $\tilde{O}(1/k^2)$ rate for smooth convex
functions, or the accelerated $\tilde{O}(\sqrt{L/\mu})$ constant in the
linear rate for strongly convex functions. Given that
the constants also depend on the dimensionality, 
it remains an open question whether this acceleration is 
practically useful.

\vspace{-3mm}
{\small \paragraph{Acknowledgements.}
We thank J.B. Alayrac, E. Hazan, A. Hubard, A. Osokin and P. Marcotte
for helpful discussions. This work was partially supported by the MSR-Inria Joint Center
and a Google Research Award.}
\bibliographystyle{abbrvnat} 
{\small
\bibliography{references}
}%

\newpage
\appendix

\setlength{\floatsep}{\Sfloatsep}
\setlength{\textfloatsep}{\Stextfloatsep}
\setlength{\intextsep}{\Sintextsep}

\rule{\textwidth}{1pt}
~\\
{\huge Appendix}

\paragraph{Outline.}
The appendix is organized as follows: In Appendix \ref{sec:FWvariants}, we discuss some of the Frank-Wolfe algorithm variants in more details and related work (including the MNP algorithm in Appendix~\ref{sec:MNPdetails} and its application to submodular minimization in Appendix \ref{app:submodular}).

In Appendix \ref{sec:pyramidal}, we discuss the pyramidal width for some particular cases of sets (such as the probability simplex and the unit cube), and then provide the proof of the main Theorem~\ref{thm:muFdirWinterpretation} relating the pyramidal width to the progress quantity essential for the linear convergence rate.
Section~\ref{sec:invariance} presents an affine invariant version of the complexity constants.

In the following Section~\ref{sec:conv}, we show the main linear convergence result for the four variants of the FW algorithm, and also discuss the sublinear rates for general convex functions.
Section~\ref{app:triangle} presents a simple experiment demonstrating the empirical tightness of the theoretical linear convergence rate constant.
Finally, in Appendix~\ref{app:NonStronglyConvex} we discuss the generalization of the linear convergence to some cases of non-strongly convex functions in more details.

\section{More on Frank-Wolfe Algorithm Variants}\label{sec:FWvariants}

\subsection{Wolfe's Min-Norm Point (MNP) algorithm}\label{sec:MNPdetails}

A generalization of Wolfe's min-norm point (MNP) algorithm~\citep{Wolfe:1976:MNP} for general convex functions is to 
run Algorithm~\ref{alg:FCFW} with the correction subroutine in step~7
implemented as presented below in Algorithm~\ref{alg:MNP}.
In Wolfe's paper~\citep{Wolfe:1976:MNP}, the correction step is called the minor cycle; whereas the FW outer loop is called the major cycle.

As we have mentioned in Section \ref{sec:variants}, MNP for polytope distance is often confused with fully-corrective FW as presented in Algorithm~\ref{alg:FCFW}, for quadratic
objectives. In fact, standard FCFW optimizes $f$ over
$\conv(\Vertices^{(t)})$, whereas MNP implements the correction as a
sequence of \emph{affine} projections on the active set that potentially yield a different update.

\begin{algorithm}
	\caption{Generalized version of Wolfe's MNP correction: \textbf{MNP-Correction}$(\x^{(t)}, \Vertices^{(t)}, \s_t)$} %
	\label{alg:MNP}
	\begin{algorithmic}[1]
	\STATE Let $\S^{(0)} := \Vertices^{(t)} \bigcup \{\s_t\}$, and $\z_0 := \x^{(t)}$. Note that $\x^{(t)} = \sum_{\vv \in \Vertices^{(t)}} \alpha_{\vv} \, \vv$ and we assume that the elements 
	of $\Vertices^{(t)}$ are \emph{affinely independent}.
	\FOR{$k=1\dots |\S^{(0)}|$}
		\STATE Let $\y_k$ be the minimizer of $f$ on the affine hull of $\S^{(k-1)}$
		\IF{$\y_k$ is in the relative interior of $\conv(\S^{(k-1)})$}
			\STATE \textbf{return} $(\y_k, \S^{(k-1)}) \quad\quad$ \emph{($\S^{(k-1)}$ is active set for $\y_k$)} 
		\ELSE
			\STATE Let $\z_k$ be the solution of doing line-search from $\z_{k-1}$ to $\y_k$. {\citepsup[step 2(d)(iii) of Alg.~1 in][]{Chakrabarty:2014:MNP}}
			\STATE \emph{\small (Note that $\z_k$ now lies on the boundary of $\conv(\S^{(k-1)})$, and so some atoms were removed)} 
			\STATE Let $\S^{(k)}$ be the (affinely independent) active atoms in the expansion of $\z_k$.
		 \ENDIF	
	\ENDFOR
	\end{algorithmic}
\end{algorithm}

There are two main differences between FCFW and the
MNP algorithm. First, after a correction step, MNP guarantees that $\x^{(t+1)}$ is \emph{both} the minimizer of
$f$ over the \emph{affine hull} of $\Vertices^{(t+1)}$ and also $\conv(\Vertices^{(t+1)})$ (where $\Vertices^{(t+1)}$ might
be much smaller than $\Vertices^{(t)} \cup
\{\s_t\}$), whereas FCFW guarantees that
$\x^{(t+1)}$ is the minimizer of $f$ over $\conv(\Vertices^{(t)} \cup
\{\s_t\})$ -- this is usually not the case for MNP unless at most one atom
was dropped from the correction polytope, as is apparent from our convergence proof. 
Secondly, the correction atoms $\Vertices^{(t)}$ are always
affinely independent for MNP and are identical to the active set
$\Coreset^{(t)}$, whereas FCFW can use both redundant as well as inactive atoms. 
The advantage of the MNP
implementation using affine hull projections is that the correction can be
efficiently implemented when $f$ is the Euclidean norm, especially when a
triangular array representation of the active set is maintained (see the
careful implementation details in Wolfe's original paper~\citep{Wolfe:1976:MNP}).

The MNP variant indeed only makes sense when the minimization of $f$ over the affine hull of $\domain$ is well-defined (and is efficient). Note though that the line-search in step 7 does not require any new information about $\domain$, as it is made only with respect to $\conv(\S^{(k-1)})$, for which we have an explicit list of vertices. This line-search can be efficiently computed in $O(|\S^{(k-1)}|)$, and is well described for example in step 2(d)(iii) of Algorithm 1 of \citetsup{Chakrabarty:2014:MNP}.

\subsection{Applications to Submodular Minimization} \label{app:submodular}
An interesting consequence of our global linear convergence result for FW algorithm variants here is the potential to reduce the gap between the known
theoretical rates and the impressive empirical performance of MNP for submodular
function minimization (over the base polytope). 
While \citet{Bach:2013et} already showed convergence of FW in this case, \citetsup{Chakrabarty:2014:MNP} later gave a weaker convergence rate for Wolfe's MNP variant.
For exact submodular function optimization, the overall complexity by \citepsup{Chakrabarty:2014:MNP} was $O(d^5 F^2)$ 
(with some corrections\footnote{\citetsup{Chakrabarty:2014:MNP} quoted a complexity of $O(d^7 F^2)$ for MNP.
However, this fell short of the earlier result of \citet{Bach:2013et} for
classic FW in the submodular minimization case, which was better by two~$O(d)$ factors. \citepsup{Chakrabarty:2014:MNP} counted $O(d^3)$ per iteration
of the MNP algorithm whereas Wolfe had provided a $O(d^2)$ implementation;
and they missed that there were at least $t/2$ good cycles (`non drop steps')
after $t$ iterations, rather than $O(t/d)$ as they have used.
}), where
$F$ is the maximum absolute value of the integer-valued submodular
function.
This is in contrast to $O(d^5 \log(d \, F))$ for the fastest
algorithms~\citepsup{Iwata:2002:subraey}. 
Using our linear convergence, the $F$ factor can be put back in the $\log$ term for 
MNP,\footnote{This is assuming that the eccentricity of the base polytope
does not depend on~$F$, which remains to be proven.}
matching their empirical observations that the MNP algorithm was not too 
sensitive to $F$. The same follows for AFW and FCFW, which is novel.

\subsection{Pairwise Frank-Wolfe}\label{sec:PFWdetails}
Our new analysis of the pairwise Frank-Wolfe variant as introduced in Section
\ref{sec:variants} is motivated by the work of~\citet{Garber:2013vl}, who
provided the first variant of Frank-Wolfe with a global linear convergence
rate with explicit constants that do not depend on the location of the
optimum $\x^*$, for a more complex extension of such a pairwise algorithm.
An important contribution of the work of~\citet{Garber:2013vl} was to define
the concept of \emph{local linear oracle}, which (approximately) minimizes a
linear function on the intersection of $\domain$ and a small ball around
$\x^{(t)}$ (hence the name \emph{local}). They showed that if such a local
linear oracle was available, then one could replace the step that moves
towards $\s_t$ in the standard FW procedure with a constant step-size move
towards the point returned by the local linear oracle to obtain a globally
linearly convergent algorithm. They then demonstrated how to implement such a
local linear oracle by using only one call to the linear oracle (to get
$\s_t$), as well as sorting the atoms in $\Coreset^{(t)}$ in decreasing order
of their inner product with $\nabla f(\x^{(t)})$ (note that the first element
then is the away atom $\vv_t$ from Algorithm~\ref{alg:AFW}). The procedure
implementing the local linear oracle amounts to iteratively swapping the mass
from the away atom~$\vv_t$ to the FW atom~$\s_t$ until enough mass has been
moved (given by some precomputed constants). If the amount of mass to move is
bigger than $\alpha_{\vv_t}^{(t)}$, then one sets $\alpha_{\vv_t}^{(t)}$  to zero and
start moving mass from the \emph{second} away atom, and so on, until enough
mass has been moved (which is why the sorting is needed). We call such a swap
of mass between the away atom and the FW atom a \emph{pairwise FW} step, i.e.
$\alpha_{\vv_t}^{(t+1)} = \alpha_{\vv_t}^{(t)} - \stepsize$ and
$\alpha_{\s_t}^{(t+1)} = \alpha_{\s_t}^{(t)} + \stepsize$ for some step-size
$\stepsize \leq \stepmax := \alpha_{\vv_t}^{(t)}$. The local linear oracle is
implemented as a sequence of pairwise FW steps, always keeping the same FW
atom~$\s_t$ as the target, but updating the away atom to move from as we set their
coordinates to zero.

A major disadvantage of the algorithm presented by~\citet{Garber:2013vl} is
that their algorithm is \emph{not adaptive}: it requires the computation of
several (loose) constants to determine the step-sizes, which means that the
behavior of the algorithm is stuck in its worst-case analysis. 
The pairwise Frank-Wolfe variant is obtained by simply doing one line-search
in the pairwise Frank-Wolfe direction $\dd^\PFW_t := \s_t - \vv_t$ (see
Algorithm~\ref{alg:PFW}). This gives a fully adaptive algorithm, and it turns
out that this is sufficient to yield a global linear convergent rate.

\paragraph{Notes on Convergence Proofs in~\cite{Nanculef:2014bj}.}
We here point out some corrections to the convergence proofs given in~\cite{Nanculef:2014bj}
for a variant of pairwise FW that chooses between a standard FW step
and a pairwise FW step by picking the one which makes the most progress
on the objective after a line-search. \cite[Proposition~$1$]{Nanculef:2014bj}
states the global convergence of their algorithm by arguing that
$\innerProdCompressed{-\nabla f(\x^{(t)})}{\dd_t^\PFW}  \geq \innerProdCompressed{-\nabla f(\x^{(t)})}{\dd_t^\FW}$ and then stating that they can re-use the same pattern 
as the standard FW
convergence proof but with the direction $\dd_t^\PFW$. But this
is forgetting the fact that the maximal step-size $\stepmax = \alpha_{\vv_t}$ for
a pairwise FW step can be too small to make sufficient progress.
Their global convergence statement is still correct as every step
of their algorithm
makes more progress than a FW step, 
which already has a global convergence result, 
but this is not the argument they made. Similarly, 
they state a global linear convergence result in their Proposition~4,
citing a proof from~\citepsup{allende2013PFW}.
On the other hand, the relevant used Proposition~3 in~\citepsup{allende2013PFW}
forgot to consider the possibility of problematic \emph{swap steps} that 
we had to painfully bound in our convergence 
Theorem~\ref{thm:megaConvergenceTheorem2}; they only
considered drop steps or `good steps', thereby missing a bound on the number of swap steps to
get a valid global bound.

\subsection{Other Related Work} \label{app:PenaDiscussion}
Very recently, following the earlier workshop version of our article~\citep{lacoste2013affine}, \citet{Pena:2015ta} presented an alternative geometric quantity measuring the linear convergence speed of the AFW algorithm variant. Their approach is motivated by a special case of the Frank-Wolfe method, the von Neumann algorithm.
Their complexity constant -- called the restricted width -- is also bounded away from zero, but its value does depend on the location of the optimal solution, which is a disadvantage shared with the earlier existing results of~\citep{Wolfe:1970wy,Guelat:1986fq,Beck:2004jm}, as well as the line of work 
of~ \citep{Ahipasaoglu:2008il,Kumar:2010ku,Nanculef:2014bj} that relies on
Robinson's condition \citep{Robinson:1982ii}.
More precisely, the bound on the constant given in~\citep[Theorem 4]{Pena:2015ta}  applies to the translated atoms $\tilde A$ relative to the optimum point. The  constant is not affine-invariant, whereas the constants $\strongConvAFW$~\eqref{eq:muf} and $\CfAFW$~\eqref{eq:CfAFW} in our setting are so, see the discussion in Section~\ref{sec:invariance}. It would still be interesting to compare
the value of our respective constants on standard polytopes.

\section{Pyramidal Width}\label{sec:pyramidal}

\subsection{Pyramidal Width of the Cube and Probability Simplex} \label{app:cubeWidth}
\begin{replemma}{lem:cubeWidth}
The pyramidal width of the unit cube in $\R^d$ is $1/\sqrt{d}$.
\end{replemma}
\begin{proof}[Proof of Lemma \ref{lem:cubeWidth}]
First consider a point $\x$ in the interior of the cube, and let $\r$ be the
unit length direction achieving the smallest pyramidal width for $\x$. Let $\s
= \s(\Vertices, \r)$ (the FW atom in direction~$\r$). Without loss of generality, 
by symmetry,\footnote{We thank Jean-Baptiste Alayrac for inspiring us
to use symmetry in the proof.} we can rotate
the cube so that $\s$ lies at the origin. This implies that each coordinate
of $\r$ is non-positive. Represent a vertex $\vv$ of the cube  as its set of
indices for which $v_i = 1$. Then $\innerProd{\r}{\s - \vv} = \sum_{i \in
\vv} -r_i \geq \max_{i \in \vv} |r_i|$. Consider any possible active set
$\S$; as $\x$ has all its coordinate strictly positive, for each dimension
$i$, there must exist an element of $\S$ with its $i$ coordinate equals to 1.
This means that $\max_{\vv \in \S} \innerProdCompressed{\r}{\s-\vv} \geq
\|\r\|_{\infty}$. But as~$\r$ has unit Euclidean norm, then
$\|\r\|_{\infty} \geq 1/\sqrt{d}$. Now consider $\x$ to lie on a facet of
the cube (i.e. the active set $\S$ is lower dimensional); and let $I := \{i :
r_i < 0 \}$. Since $\r$ has to be feasible from $\x$, for each $i \in I$, we
cannot have $x_i = 0$ and thus there exists an element of the active set with
its $i^{\text{th}}$ coordinate equal to 1. We thus have that $\max_{\vv \in
\S} \innerProdCompressed{\r}{\s-\vv} \geq \|\r\|_{\infty} \geq 1/\sqrt{|I|}
\geq 1/\sqrt{d}$. Using the same argument on a lower dimensional $\Kface$
give a lower bound of $1/\sqrt{\dim(\Kface)}$ which is bigger. These cover
all the possibilities appearing in the definition of the pyramidal width, and
thus the lower bound is correct. It is achieved by choosing an $\x$ in the
interior, the canonical basis as the active set $\S$, and the direction
defined by $r_i = -1/\sqrt{d}$ for each $i$. 
\end{proof}

We note that both the active set definition $\S$ and the feasibility condition on $\r$
were crucially used in the above proof to obtain such a large value
for the pyramidal width of the unit cube, thus justifying the somewhat
involved definition appearing in~\eqref{eq:Pwidth}. On the other hand, 
the astute reader might have noticed that the important quantity to lower
bound for the linear convergence rate of the different FW variants is
$\frac{\innerProdCompressed{\r_t}{\hat{\dd}_t}}{\innerProdCompressed{\r_t}{\hat{\err}_t}}$
(as in~\eqref{eq:linearProgressBound}),
rather than the looser value $\frac{1}{M} \frac{\innerProdCompressed{\r_t}{\dd_t}}{\innerProdCompressed{\r_t}{\hat{\err}_t}}$ 
that we used to handle the proof of the difficult Theorem~\ref{thm:muFdirWinterpretation}
(where we recall that $M$ is the diameter of $\domain$).
One could thus hope to get a tighter measure for the condition number of a set
by considering $\|\s -\vv\|$ (with $\s$ and $\vv$ the minimizing witnesses for
the pyramidal width) instead of the diameter $M$ in the ratio diameter / pyramidal width.
This might give a tighter constant for general sets, but in the case of the cube, it does not
change the general $\Omega(d^2)$ dependence for its condition number. To see this, suppose that $d$ is
even and let $k = d/2$. Consider the direction $\r$ with $r_i := -1$ for $1 \leq i \leq k$, and $r_i := -\epsilon$ for $(k\!+\!1) \leq i \leq d$. We thus have that the FW atom $\s(\Vertices,\r)$
is the origin as before. Consider $\x$ such that $x_i := 1/k$ for $1 \leq i \leq k$, 
and $x_i := 1$ for $(k\!+\!1) \leq i \leq d$, that is, $\x$ is the uniform convex combination
of the $k$ vertices which has only one non-zero in the first $k$ coordinates, and the last $k$
coordinates all equal to $1$. We have that $\r$ is a feasible direction from $\x$, and
that all vertices in the active set for $\x$ have the same inner product with $\r$:
$\max_{\vv \in \S} \innerProdCompressed{\r}{\s-\vv} = 1 + k \epsilon$. We thus have:
$$
\innerProd{\frac{\r}{\|\r\|}}{\frac{\s - \vv}{\|\s - \vv\|}} = \frac{1 + k \epsilon}{(\sqrt{k}\sqrt{1 + \epsilon^2}) (\sqrt{k+1})} \leq \frac{1}{k} \quad \text{for $\epsilon$ small enough}.
$$
Squaring the inverse, we thus get that the condition number of the cube is at least $k^2 = d^2/4$
even using this tighter definition, thus not changing the $\Omega(d^2)$ dependence.

\paragraph{Pyramidal Width for the Probability Simplex.} For any $\x$ in the relative
interior of the probability simplex on $d$ vertices, we have that $\S = \Vertices$,
and thus the pyramidal directional width~\eqref{eq:TruePdirW} in the feasible direction $\r$
with base point $\x$ is the same as the standard directional width.
Moreover, any face of the probability simplex is just a probability simplex in lower
dimensions (with bigger width). This is why the pyramidal width of the
probability simplex is the same as its standard width. The width of a regular simplex
was implicitly given in~\citep{Alexander:1977:simplex}; we provide more
details here on this citation. \citet{Alexander:1977:simplex} considers a regular
simplex with $k$ vertices and side length $\Delta$. For any partition of the $k$ points
into a set of $r$ and $k-r$ points (for $r \leq \left\lfloor{k/2}\right \rfloor$),
one can compute the distance $c(r,\Delta)$ between the flats (affine hulls) of the two
sets (which also corresponds to a specific directional width). 
Alexander gives a formula on the third line of p.~91 
in~\citep{Alexander:1977:simplex} for the square of this distance:
\begin{equation} \label{eq:flatDistance}
\left( c(r,\Delta) \right)^2 = \Delta^2 \frac{k}{2r (k-r)} . 
\end{equation}
The width of the regular simplex is obtained by taking the minimum of~\eqref{eq:flatDistance} with
respect to $r  \leq \left\lfloor{k/2}\right \rfloor$. As~\eqref{eq:flatDistance} is
a decreasing function up to $r = k/2$, we obtain its minimum 
by substituting $r = \left\lfloor{k/2}\right \rfloor$. By using $\Delta = \sqrt{2}$, $k=d$
and $r = \left\lfloor{d/2}\right \rfloor$ in~\eqref{eq:flatDistance},
we get that the width for the probability simplex on $d$ vertices is $2/\sqrt{d}$ when
$d$ is even and the slightly bigger $2/\sqrt{d - 1/d}$ when $d$ is odd.

From the pyramidal width perspective, one can obtain these numbers by 
considering any relative interior point of the probability simplex as $\x$
and considering the following feasible $\r$. For $d$ even,
we let $r_i := 1$ for $1 \leq i \leq d/2$ and $r_i := -1$ for $i > d/2$.
Note that $\sum_i r_i = 0$ and thus $\r$ is a feasible direction
for a point in the relative interior of the probability simplex.
Then $\max_{\s, \vv \in \Vertices} \innerProdCompressed{\frac{\r}{\| \r\|}}{\s - \vv} = \frac{1}{\sqrt{d}} (1+1) = \frac{2}{\sqrt{d}}$ as claimed. For $d$ odd, we can choose
$r_i := \frac{2}{d-1}$ for $1 \leq i \leq \frac{d-1}{2}$, and $r_i := \frac{2}{d+1}$
for $i \geq \frac{d+1}{2}$. Then $\| \r \| = \sqrt{\frac{4d}{d^2 -1}}$ and
$\max_{\s, \vv \in \Vertices} \innerProdCompressed{\frac{\r}{\| \r\|}}{\s - \vv} = \sqrt{\frac{4d}{d^2-1}}
= 2/\sqrt{d - 1/d}$ as claimed. Showing that these obtained values
were the minimum possible ones is non-trivial though, which is
why we appealed to the width of the regular simplex computed 
from~\citep{Alexander:1977:simplex}.

\subsection{Proof of Theorem~\ref{thm:muFdirWinterpretation} on the Pyramidal Width} \label{app:ProofWidth} 

In this section, we prove the main technical result in our paper: a geometric lower bound
for the crucial quantity appearing in the linear convergence rate for the FW 
optimization variants.

\begin{reptheorem}{thm:muFdirWinterpretation}
Let $\x \in \domain=\conv(\Vertices)$ be a suboptimal point and $\S$ be an
active set for $\x$. Let $\x^*$ be an optimal point and corresponding error
direction $\hat{\err} = (\x^*-\x)/\norm{\x^*-\x}$, and negative gradient $\r
:= -\nabla f(\x)$  (and so $\innerProdCompressed{\r}{\hat{\err}} > 0$). Let
$\dd^\PFW$ be the pairwise FW direction obtained over $\Vertices$ and $\S$ with negative
gradient $\r$. Then we have:
\begin{equation} \label{eq:PWidthIsBound2}
\frac{\innerProdCompressed{\r}{\dd^\PFW}}{\innerProdCompressed{\r}{\hat{\err}}}
\geq \PWidth(\Vertices) .
\end{equation}   
\end{reptheorem}

We first give a proof sketch, and then give the full proof.

Recall that a direction~$\r$ is \emph{feasible} for $\Vertices$ from $\x$ if it points inwards $\conv(\Vertices)$, 
i.e. $\r \in \text{cone}(\Vertices-\x)$.

\begin{proof}[Warm-Up Proof Sketch]
By Cauchy-Schwarz, the denominator of~\eqref{eq:PWidthIsBound2} is at most
$\norm{\r}$. If $\r$ is a feasible direction from $\x$ for $\Vertices$, then the LHS is lower bounded by $\PdirW(\Vertices, \r,
\x)$ as $\dd^\PFW$ is included as a possible $\s - \vv$ direction considered in the definition of $\PdirW$~\eqref{eq:TruePdirW}.
If $\r$ is not feasible from $\x$, this means that
$\x$ lies on the boundary of $\domain$. One can then show that the
potential $\x^*$ that can maximize $\innerProdCompressed{\r}{\hat{\err}}$ has
also to lie on a facet~$\Kface$ of~$\domain$ containing $\x$ (see Lemma~\ref{lem:minAngle} below). 
The idea is then to
project~$\r$ onto~$\Kface$, and re-do the argument with $\Kface$ replacing
$\domain$ and show that the inequality is in the right direction. This
explains why all the subfaces of $\domain$ are considered in the definition of
the pyramidal width~\eqref{eq:Pwidth}, and that only \emph{feasible}
directions are considered.
\end{proof}

\begin{lemma}[Minimizing angle is on a facet] \label{lem:minAngle}
Let $\x$ be at the origin, inside a polytope $\Kface$ and suppose that $\r \in \text{span}(\Kface)$ is not a feasible direction for $\Kface$ from $\x$ (i.e. $\r \notin \text{cone}(\Kface)$). Then a feasible direction in $\Kface$ minimizing the angle with $\r$ lies on a facet\footnote{As a reminder, we define a \emph{k-face} of
$\domain$ (a $k$-dimensional face of $\domain$) a set $\Kface$ such that
$\Kface = \domain \cap \{ \y : \innerProd{\r}{\y - \x} = \0 \}$ for some
normal vector $\r$ and fixed reference point $\x \in \Kface$ with the
additional property that $\domain$ lies on one side of the given half-space
determined by $\r$ i.e. $ \innerProd{\r}{\y - \x} \leq \0$ $\forall \y \in
\domain$. $k$ is the dimensionality of the affine hull of $\Kface$. We call a
$k$-face of dimensions $k = 0$, $1$, $\textrm{dim}(\domain)-2$ and
$\textrm{dim}(\domain)-1$ a \emph{vertex}, \emph{edge}, \emph{ridge} and
\emph{facet} respectively. $\domain$ is a $k$-face of itself with $k =
\textrm{dim}(\domain)$. See Definition 2.1 in the book of~\citetsup{Ziegler:1995td}, which we also recommend for more background material on polytopes.}%
~$\Kface'$ of $\Kface$ that includes the origin $\x$. That is:
\begin{equation} \label{eq:angleRelation}
\max_{\err \in \Kface} \innerProd{\r}{\frac{\err}{\|\err\|}} = \max_{\err \in \Kface'} \innerProd{\r}{\frac{\err}{\|\err\|}} =  \max_{\err \in \Kface'} \innerProd{\r'}{\frac{\err}{\|\err\|}}
\end{equation}
where $\Kface'$ contains $\x$, $\|\cdot \|$ is the Euclidean norm and $\r'$ is defined
as the orthogonal projection of~$\r$ on~$\text{span}(\Kface')$.
\end{lemma}

\begin{figure}[t]
\begin{center}
\includegraphics[width=0.7\linewidth]{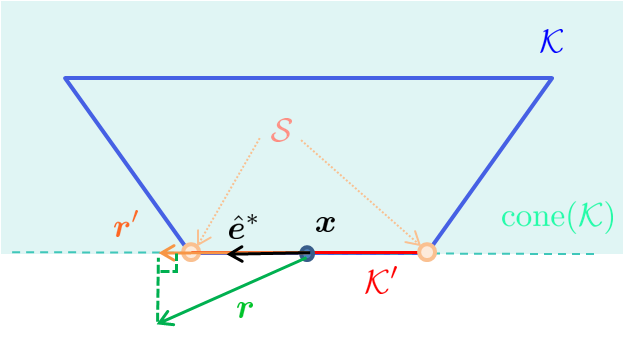}
\end{center}
\vspace{-4mm}
\caption{Depiction of the quantities in the proof of Lemma~\ref{lem:minAngle}. 
If $\r \notin \text{cone}(\Kface)$, but $\r \in \text{span}(\Kface)$, then the unit vector
direction~$\hat{\err}^*$ minimizing the angle with~$\r$ is generated
by a point~$\x^*$ lying on a facet~$\Kface'$ of the polytope~$\Kface$
that contains~$\x$. \vspace{5mm}} \label{fig:rFig}
\end{figure}

\begin{proof}
This seems like an obvious geometric fact (see Figure~\ref{fig:rFig}), but we prove it formally, as sometimes high dimensional geometry is tricky (for example, the result is false without the assumption that $\r \in \text{span}(\Kface)$ or if $\|\cdot\|$ is not the Euclidean norm). Rewrite the optimization variable on the LHS of~\eqref{eq:angleRelation} as $\hat{\err} = \frac{\err}{\|\err\|}$. The optimization domain for $\hat{\err}$ is thus the intersection between the unit sphere and $\text{cone}(\Kface)$.
We now show that any maximizer $\hat{\err}^*$ cannot lie in the relative interior of $\text{cone}(\Kface)$, and thus it has to lie on a \emph{facet} of $\text{cone}(\Kface)$, implying then that a corresponding maximizer $\err^*$ is lying on a facet of $\Kface$ containing $\x$, concluding the proof for the first equality in~\eqref{eq:angleRelation}.

First, as $\r \in \text{span}(\Kface)$, we can consider without loss of generality that $\text{cone}(\Kface)$ is full dimensional by projecting on its affine hull if needed. We want to solve $\max_{\hat{\err}} \innerProd{\r}{\hat{\err}}$ s.t. $\| \hat{\err}\|^2=1$ and $\hat{\err} \in \text{cone}(\Kface)$. By contradiction, we suppose that $\hat{\err}^*$ lies in the interior of $\text{cone}(\Kface)$, and so we can remove the polyhedral cone constraint. The gradient of the objective is the constant $\r$ and the gradient of the equality constraint is $2\hat{\err}$. By the Karush-Kuhn-Tucker (KKT) necessary conditions for a stationary point to the problem with the only equality constraint $\| \hat{\err}\|^2=1$ (see e.g.~\citepsup[Proposition 3.3.1 in][]{bertsekas1999nonlinear}), 
then the gradient of the objective is collinear to the gradient of the equality constraint, i.e. we have $\hat{\err}^* = \pm \hat{\r}$. Since $\hat{\r}$ is not feasible, then $\hat{\err}^* = -\hat{\r}$, which is actually a local \emph{minimum} of the inner product by Cauchy-Schwarz. We thus conclude that the maximizing $\hat{\err}^*$ lies on the boundary of $\text{cone}(\Kface)$, concluding the proof for the first equality in~\eqref{eq:angleRelation}.

For the second equality in~\eqref{eq:angleRelation}, we simply use the fact that $\r - \r'$ is orthogonal
to the elements of~$\Kface'$ by the definition of the orthogonal projection.
\end{proof}

\begin{proof}[Proof of Theorem~\ref{thm:muFdirWinterpretation}]
Let $\hat{\err}(\x^*) := \frac{\x^*-\x}{\norm{\x^*-\x}}$ be the normalized error vector. 
We consider the worst-case possibility for $\x^*$. As $\x$ is not optimal, we require that $\innerProd{\r}{ \hat{\err}(\x^*)} > 0$.
We recall that by definition of the pairwise FW direction:
\begin{equation} \label{eq:simplePDirWidthBound}
\innerProd{\frac{\r}{\|\r\|}}{\dd^\PFW} = \max_{\s \in \Vertices, \vv \in \S} \innerProd{\frac{\r}{\|\r\|}}{\s - \vv} 
\geq \min_{\S' \in \S_{\x}} \max_{\s \in \Vertices, \vv \in \S'} \innerProd{\frac{\r}{\|\r\|}}{\s - \vv} = \PdirW(\Vertices, \r, \x). 
\end{equation}
By Cauchy-Schwarz, we always have $\innerProd{\r}{\hat{\err}(\x^*)} \leq \|\r\|$. 
If $\r$ is a feasible direction from~$\x$ in~$\conv(\Vertices)$, then $\r$ appears in the set of directions
considered in the definition of the pyramidal width~\eqref{eq:Pwidth} for $\Vertices$ and so from~\eqref{eq:simplePDirWidthBound}, we have that the inequality~\eqref{eq:PWidthIsBound2} holds.

If $\r$ is not a feasible direction, then we iteratively project it on the faces of $\domain$ until we get a feasible direction $\r'$, obtaining a term $\PdirW(\Vertices \cap \Kface, \r', \x)$ for some 
face $\Kface$ of $\domain$ as appearing in the definition of the pyramidal width~\eqref{eq:Pwidth}. 
The rest of the proof formalizes this process. As $\x$ is fixed, 
we work on the centered polytope at $\x$ to simplify the statements, i.e. 
let $\tilde{\domain} := \domain - \x$. We have
the following worst case lower bound for~\eqref{eq:PWidthIsBound2}:
\begin{equation} \label{eq:MainLowerBound}
\frac{\innerProdCompressed{\r}{\dd^\PFW}}{\innerProdCompressed{\r}{\hat{\err}}} \geq
	\bigg( \max_{\s \in \Vertices, \vv \in \S} \innerProdCompressed{\r}{\s - \vv} \bigg)
	\left( \max_{\err \in \tilde{\domain}} \, \innerProdCompressed{\r}{\frac{\err}{\|\err\|}} \right)^{-1} . 
\end{equation}
The first term on the RHS of~\eqref{eq:MainLowerBound} just comes from the definition of
$\dd^\PFW$ (with equality), whereas the second term is considering the worst case
possibility for $\x^*$ to lower bound the LHS. Note also that the second term has to
be strictly greater to zero since $\x$ is not optimal.

Without loss of generality, we can assume that $\r \in \text{span}(\tilde{\domain})$ (otherwise, just project it),
as any orthogonal component would not change the inner products appearing in~\eqref{eq:MainLowerBound}.
If (this projected) $\r$ is feasible from $\x$, then $\max_{\err \in \tilde{\domain}} \, \innerProdCompressed{\r}{\frac{\err}{\|\err\|}} = \| \r \|$, and we again
have the lower bound~\eqref{eq:simplePDirWidthBound} arising in the definition
of the pyramidal width.

We thus now suppose that $\r$ is not feasible. By the Lemma~\ref{lem:minAngle},
we have the existence of a facet~$\Kface'$ of~$\tilde{\domain}$
that includes the origin $\x$ such that:
\begin{equation} \label{eq:angleEquality}
\max_{\err \in \tilde{\domain}} \innerProd{\r}{\frac{\err}{\|\err\|}} = \max_{\err \in \Kface'} \innerProd{\r}{\frac{\err}{\|\err\|}} = \max_{\err \in \Kface'} \innerProd{\r'}{\frac{\err}{\|\err\|}},
\end{equation}
where~$\r'$ is the result of the orthogonal projection of~$\r$ on~$\text{span}(\Kface')$.
We now look at how the numerator of~\eqref{eq:MainLowerBound} transforms
when considering $\r'$ and $\Kface'$:
\begin{align}
	\max_{\s \in \Vertices, \vv \in \S} \innerProdCompressed{\r}{\s - \vv}   &= 
	 \max_{\s \in \domain} \innerProd{\r}{ \s - \x} +
     \max_{\vv \in \S} \innerProd{-\r}{ \vv - \x} 
\nonumber \\
	 &\ge \max_{\s \in (\Kface'+\x)} \innerProd{\r}{  \s - \x } +
 \max_{\vv \in \S \cap (\Kface' + \x)}
\innerProd{-\r}{ \vv - \x} \nonumber \\
	 &=  \max_{\s \in (\Kface'+\x)} \innerProd{\r'}{  \s - \x } +
	 	  \max_{\vv \in \S } \innerProd{-\r'}{
\vv - \x} \nonumber \\
     &= \max_{\s \in \Vertices \cap (\Kface' + \x), \vv \in \S} \innerProd{\r'}{\s - \vv} .
     \label{eq:d1sf}
\end{align}
To go from the first to the second
line, we use the fact that the first term yields an inequality as $(\Kface'+\x)
\subseteq (\tilde{\domain}+ \x) = \domain$. Also, since 
$\x$ is in the relative interior of $\conv(\S)$ (as $\x$ is a \emph{proper} 
convex combination of elements of $\S$ by definition), we have that
$(\S - \x) \subseteq \Kface$ for any face~$\Kface$ of~$\tilde{\domain}$
containing the origin $\x$. Thus $\S = \S\cap (\Kface'+\x)$, and the
second term on the first line actually yields an equality for the second line. 
The third line uses the
fact that $\r - \r'$ is orthogonal to members of $\Kface'$,
as $\r'$ is obtained by orthogonal projection.

Plugging~\eqref{eq:angleEquality} and~\eqref{eq:d1sf} into the inequality~\eqref{eq:MainLowerBound},
we get:
\begin{equation} \label{eq:MainLowerBound2}
\frac{\innerProdCompressed{\r}{\dd^\PFW}}{\innerProdCompressed{\r}{\hat{\err}}} \geq
	\bigg( \max_{\substack{\s \in \Vertices \cap (\Kface' + \x), \\ \vv \in \S}} \innerProdCompressed{\r'}{\s - \vv} \bigg)
	\left( \max_{\err \in \Kface'} \, \innerProdCompressed{\r'}{\frac{\err}{\|\err\|}} \right)^{-1} .
\end{equation}
We are back to a similar situation to~\eqref{eq:MainLowerBound}, with the lower dimensional~$\Kface'$ playing the role of the polytope~$\tilde{\domain}$, and $\r' \in \text{span}(\Kface')$ playing the role of $\r$.
If $\r'$ is feasible from $\x$ in $\Kface'$, then re-using the previous argument,
we get $\PdirW(\Vertices \cap (\Kface' + \x), \r', \x)$ as the lower bound, which is part
of the definition of the pyramidal width of $\Vertices$ (note that we have $(\Kface'\!+\!\x)$
as $\Kface'$ is a face of the \emph{centered} polytope $\tilde{\domain}$). 
Otherwise (if $\r \not \in \text{cone}(\Kface')$),
then we use Lemma~\ref{lem:minAngle} again to get a facet~$\Kface''$ of~$\Kface'$ as
well as a new direction~$\r''$ which is the orthogonal projection of~$\r'$ on~$\text{span}(\Kface'')$
such that we can re-do the manipulations for~\eqref{eq:angleEquality} and~\eqref{eq:MainLowerBound2},
yielding $\PdirW(\Vertices \cap (\Kface'' + \x), \r'', \x)$ as a lower bound
if~$\r''$ is feasible from~$\x$ in~$\Kface''$. As long as we do not obtain a feasible
direction, we keep re-using Lemma~\ref{lem:minAngle} to project the direction
on a lower dimensional face of
$\tilde{\domain}$ that contains $\x$. This process must stop at some point; ultimately,
we will reach the lowest dimensional face~$\Kface_{\x}$ that contains~$\x$.
As~$\x$ lies in the relative interior of~$\Kface_{\x}$, then all 
directions in~$\text{span}(\Kface_{\x})$ are feasible, and so the projected
$\r$ will have to be feasible.
Moreover, by stringing together the equalities of the type~\eqref{eq:angleEquality}
for all the projected directions, we know that $\max_{\err \in \Kface_{\x}} \innerProdCompressed{\r_{\mathrm{final}}}{\frac{\err}{\|\err\|}} > 0$ (as we originally
had $\innerProdCompressed{\r}{\hat{\err}} > 0$), and thus $\Kface_{\x}$ is 
at least one-dimensional and we also have $\r_{\mathrm{final}} \neq \0$
(this last condition is crucial to avoid having a lower bound of zero!).
This concludes the proof, and also explains why in the definition
of the pyramidal width~\eqref{eq:Pwidth}, we consider
the pyramidal directional width for
all the faces of $\conv(\Vertices)$ and respective non-zero 
feasible direction $\r$.
\end{proof}

\section{Affine Invariant Formulation}\label{sec:invariance}

Here we provide linear convergence proofs in terms of affine invariant quantities, since all the Frank-Wolfe algorithm variants presented in this paper are affine invariant. The statements presented in the main paper above are special cases of the following more general theorems, by using the bounds~\eqref{eq:CfBound} for the curvature constant $\Cf$, and Theorem~\ref{thm:muFdirWinterpretation2} for the affine invariant strong convexity~$\strongConvAFW$.

An optimization method is called \emph{affine invariant} if it is invariant
under affine transformations of the input problem: If one chooses any
re-parameterization of the domain~$\domain$ by a \emph{surjective} linear or
affine map $\A:\hat\domain\rightarrow\domain$, then the ``old'' and ``new''
optimization problems $\min_{\x\in\domain}f(\x)$ and
$\min_{\hat\x\in\hat\domain}\hat f(\hat\x)$ for $\hat f(\hat\x):=f(\A\hat\x)$
look completely the same to the algorithm.

More precisely, every ``new'' iterate must remain exactly the transform of
the corresponding old iterate; an affine invariant analysis should thus yield
the convergence rate and constants unchanged by the transformation. It is
well known that Newton's method is affine invariant under invertible~$\A$,
and the Frank-Wolfe algorithm and all the variants presented here are affine
invariant in the even stronger sense under arbitrary surjective
$\A$~\citep{Jaggi:2013wg}. (This is directly implied if the algorithm and all
constants appearing in the analysis only depend on inner products with the
gradient, which are preserved since $\nabla \hat f = \A^T\nabla f$.)

Note however that the property of being an extremum point (vertex) of~$\domain$ is
\emph{not} affine invariant (see~\citep[Section~3.1]{Beck:2015vo} for an
example). This explains why we presented all algorithms here as
working with atoms $\Vertices$ rather than vertices of the domain, thus maintaining the affine invariance of the algorithms as well as their convergence analysis.

\paragraph{Affine Invariant Measures of Smoothness.}
The affine invariant convergence analysis of the standard Frank-Wolfe
algorithm by \citep{Jaggi:2013wg} crucially relies on the following measure
of non-linearity of the objective function $f$ over the domain $\domain$. The
(upper) \emph{curvature constant} $C_{f}$ of a convex and differentiable
function $f:\R^d\rightarrow\R$, with respect to a compact domain $\domain$ is defined as
\begin{equation}\label{eq:Cf}
  \Cf := \sup_{\substack{\x,\s\in \domain,  ~\stepsize\in[0,1],\\
                      \y = \x+\stepsize(\s-\x)}} \textstyle
           \frac{2}{\stepsize^2}\big( f(\y)-f(\x)-\innerProd{\nabla f(\x)}{\y-\x}\big) \ .
\end{equation}

The definition of $\Cf$ closely mimics the fundamental descent
lemma~\eqref{eq:descentLemma}. 
The assumption of bounded curvature $\Cf$ closely corresponds to a Lipschitz
assumption on the gradient of~$f$. %
More precisely, if~$\nabla f$ is $L$-Lipschitz continuous on $\domain$ with
respect to some arbitrary chosen norm $\norm{.}$ in dual pairing, i.e.
$\norm{\nabla f(\x) - \nabla f(\y)}_* \leq L \norm{\x-\y}$, then
\begin{equation} \label{eq:CfBound}
\Cf \le L \diam_{\norm{.}}(\domain)^2  \ ,
\end{equation}
where $\diam_{\norm{.}}(.)$ denotes the $\norm{.}$-diameter, see \citep[Lemma
7]{Jaggi:2013wg}.
While the early papers \citep{Frank:1956vp,Dunn:1979da} on the Frank-Wolfe
algorithm relied on such Lipschitz constants with respect to a norm, 
the curvature constant $\Cf$ here is affine invariant, does not depend on any
norm, and gives tighter convergence rates. 
The quantity $\Cf$ combines the complexity of the domain $\domain$ and the curvature of the objective function $f$ into a
single quantity. The advantage of this combination is well illustrated in~\citep[Lemma A.1]{LacosteJulien:2013ue}, where Frank-Wolfe was used to
optimize a quadratic function over product of probability simplices with an
exponential number %
of dimensions. In this case, the Lipschitz constant could
be exponentially worse than the curvature constant which does take the simplex 
geometry of $\domain$ into account.

\paragraph{An Affine Invariant Notion of Strong Convexity which Depends on the
Geometry of $\domain$.} We now present the affine invariant analog of the
strong convexity bound~\eqref{eq:strongLower}, which could be interpreted as
the \emph{lower} curvature $\strongConvAFW$ analog of $\Cf$. The role of
$\stepsize$ in the definition~\eqref{eq:Cf} of~$\Cf$ was to define an affine
invariant scale by only looking at proportions over lines (as line segments
between $\x$ and~$\s$ in this case). The trick here is to use anchor points
in~$\Vertices$ in order to define standard lengths (by looking at proportions
on lines). These anchor points ($\s_f(\x)$ and $\vv_f(\x)$ defined below) are
motivated directly from the FW atom and the away atom appearing in the away-steps FW 
algorithm. Specifically, let $\x^*$ be a potential optimal point
and $\x$ a non-optimal point; thus we have $\innerProd{-\nabla f(\x)}{\x^*-\x} > 0$ 
(i.e. $\x^*\!-\!\x$ is a strict descent direction from $\x$ for $f$). 
We then define the positive step-size quantity:
\begin{equation}\label{eq:gammaAway}
\stepsize^\away(\x,\x^*) := \frac{\innerProd{-\nabla f(\x)}{\x^*-\x}
}{\innerProd{-\nabla f(\x)}{\s_f(\x) - \vv_f(\x)} } \ . 
\end{equation}
This quantity is
motivated from both~\eqref{eq:AFWgapInequality} and the linear rate
inequality~\eqref{eq:linearProgressBound}, and enables to transfer lengths
from the error $\err_t = \x^* - \x_t$ to the pairwise FW direction $\dd_t^\PFW = \s_f(\x_t) - \vv_f(\x_t)$.
More precisely, $\s_f(\x) :=$%
$~\argmin_{\vv \in \Vertices} \left\langle \nabla f(\x), \vv \right\rangle$ is the
standard FW atom. To define the away-atom, we consider all possible expansions of 
$\x$ as a convex combination of atoms.\footnote{As we are working with general polytopes,
the expansion of a point as a convex combination of atoms is not 
necessarily unique.}
We recall that set of possible active sets is $\S_{\x} := \{ \S \, | \, \S \subseteq \Vertices$ such that $\x$ is a
proper convex combination of
all the elements in $\S\}$.
For a given set~$\S$, we write $\vv_{\S}(\x) := \argmax_{\vv \in \S }
\left\langle \nabla f(\x), \vv \right\rangle$ for the away atom in the
algorithm supposing that the current set of active atoms is $\S$.
Finally, we define $\vv_f(\x) := \hspace{-3mm}\displaystyle\argmin_{\{\vv = \vv_{\S}(\x)
\,|\, \S \in \S_{\x} \}} \textstyle \hspace{-3mm}\left\langle \nabla f(\x), \vv
\right\rangle$ to be the worst-case away atom (that is, the atom which would
yield the smallest away descent).

We then define the \emph{geometric strong convexity} constant
$\strongConvAFW$ which depends \emph{both} on the function~$f$ and the
domain~$\domain=\conv(\Vertices)$:
\begin{equation}\label{eq:muf}
  \strongConvAFW \ :=\  \inf_{\x\in \domain} \inf_{\substack{\x^* \in \domain\\
                        \textrm{s.t. } \left\langle \nabla f(\x), \x^*-\x \right\rangle < 0 }}
           \frac{2}{{\stepsize^\away(\x,\x^*)}^2}
           \big( f(\x^*)-f(\x)-\left\langle \nabla f(\x),  \x^*-\x \right\rangle \big) \ .
\end{equation}

\subsection{Lower Bound for the Geometric Strong Convexity Constant $\strongConvAFW$}
The geometric strong convexity constant $\strongConvAFW$, as defined in
(\ref{eq:muf}), is affine invariant, since it only depends on the inner
products of feasible points with the gradient. Also, it combines both the
complexity of the function~$f$ and the geometry of the domain $\domain$.
Theorem~\ref{thm:muFdirWinterpretation2} allows us to lower bound the constant
$\strongConvAFW$ in terms of the strong convexity of the objective function,
combined with a purely geometric complexity measure of the domain $\domain$
(its pyramidal width $\PWidth(\Vertices)$~\eqref{eq:Pwidth}).
In the following Section~\ref{sec:conv} below, we will show the linear
convergence of the four variants of the FW algorithm presented in this paper under the assumption that
$\strongConvAFW > 0$. 

In view of the following Theorem~\ref{thm:muFdirWinterpretation2}, we have 
that the condition $\strongConvAFW > 0$ is slightly weaker
 than the strong 
convexity of the objective function\footnote{As an example
of function that is not strongly convex but can still have $\strongConvAFW > 0$,
consider $f(\x) := g(\A \x)$ where $g$ is $\mu_g$-strongly convex, but the
matrix $\A$ is rank deficient. Then by using the affine invariance
of the definition of~$\strongConvAFW$ and 
using Theorem~\eqref{thm:muFdirWinterpretation2} applied
on the equivalent problem on~$g$ with domain~$\conv(\A \Vertices)$,
we get $\strongConvAFW \geq \mu_g \cdot (\PWidth(\A \Vertices))^2 > 0$.
} over a polytope domain 
(it is implied by strong convexity).

\begin{theorem}\label{thm:muFdirWinterpretation2}
Let $f$ be a convex differentiable function and suppose that $f$ is
$\mu$-\emph{strongly convex} w.r.t. to the Euclidean norm
$\norm{\cdot}$ over the domain $\domain=\conv(\Vertices)$ with strong-convexity constant $\mu
\geq 0$. Then
\begin{equation} \label{eq:muDirWidthBound}
\strongConvAFW \geq \mu \cdot \left( \PWidth(\Vertices) \right)^2 \ .
\end{equation}
\end{theorem}
\begin{proof}
By definition of strong convexity with respect to a norm, we have that for any $\x,\y\in\domain$,
\begin{equation} \label{eq:strongConv}
f(\y)- f(\x)- \langle\nabla f(\x), \y-\x \rangle
\geq \textstyle\frac{\mu}{2} \norm{\y-\x}^2 \ .
\end{equation}
Using the strong convexity bound~\eqref{eq:strongConv} with $\y := \x^*$ on the right hand side of equation~\eqref{eq:muf} (and using the shorthand $\r_{\x} := -\nabla f(\x)$ ), we thus get:
\begin{align}
\strongConvAFW \geq&  \inf_{\substack{\x, \x^* \in \domain\\
                                   \textrm{s.t. } \innerProd{\r_{\x}}{\x^*-\x} > 0}}
                      \mu \left(  \frac{\innerProd{\r_{\x}}{ \s_f(\x) - \vv_f(\x)}}{\innerProd{\r_{\x}}{\x^*-\x}} \norm{{\x^*-\x}} \right)^2 \nonumber \\
		&=  \mu \inf_{\substack{\x \neq \x^* \in \domain \\
                        \textrm{s.t. } \innerProd{\r_{\x}}{ \hat{\err}} > 0}}
           \left(  \frac{\innerProd{\r_{\x}}{ \s_f(\x) - \vv_f(\x)}}{\innerProd{\r_{\x}}{ \hat{\err}(\x^*,\x)}} \right)^2  , \label{eq:mufInitial}
\end{align}
where $\hat{\err}(\x^*,\x) := \frac{\x^*-\x}{\norm{\x^*-\x}}$ is the unit length feasible
direction from $\x$ to $\x^*$. We are thus taking an infimum over all
possible feasible directions starting from $\x$ (i.e. which moves within
$\domain$) with the additional constraint that it makes a positive inner
product with the negative gradient $\r_{\x}$, i.e. it is a strict descent
direction. This is only possible if $\x$ is not already optimal, i.e. $\x \in
\domain \setminus \X^*$ where $\X^* := \{\x^* \in \domain :
\innerProd{\r_{\x^*}}{\x-\x^*} \leq 0 \,\, \forall \x \in \domain \}$ is the
set of optimal points. 

We note that $\s_f(\x) - \vv_f(\x)$ is a valid pairwise FW direction for a specific active set $\S$ for $\x$, and so we can re-use~\eqref{eq:PWidthIsBound2} from Theorem~\ref{thm:muFdirWinterpretation} for the right hand side of \eqref{eq:mufInitial} to conclude the proof.
\end{proof}

We now proceed to present the main linear convergence result in the next section, using only the mentioned affine invariant quantities.

\section{Linear Convergence Proofs}\label{sec:conv}

\paragraph{Curvature Constants.}
Because of the additional possibility of the away step in
Algorithm~\ref{alg:AFW}, we need to define the following slightly modified
additional curvature constant, which will be needed for the linear
convergence analysis of the algorithm:
\begin{equation}\label{eq:CfAFW}
  \CfAFW := \sup_{\substack{\x,\s,\vv\in \domain, \\
                      \stepsize\in[0,1],\\
                      \y = \x+\stepsize(\s-\vv)}}
           \frac{2}{\stepsize^2}\big( f(\y)-f(\x)-\stepsize \langle \nabla f(\x), \s-\vv
\rangle \big) \ .
\end{equation}
By comparing with $\Cf$~\eqref{eq:Cf}, we see that the modification is that
$\y$ is defined with any direction $\s -\vv$ instead of a standard
FW direction $\s - \x$. This allows to use the away direction or the pairwise FW direction even though these might yield some $\y$'s which are outside of the
domain $\domain$ when using $\stepsize > \stepmax$ (in fact, $\y \in \domain^\away := \domain + (\domain -
\domain)$ in the Minkowski sense). On the other hand, by re-using a similar
argument as in \citep[Lemma 7]{Jaggi:2013wg}, we can obtain the same
bound~\eqref{eq:CfBound} for $\CfAFW$, with the only difference that the
Lipschitz constant $L$ for the gradient function has to be valid on
$\domain^\away$ instead of just $\domain$.

\begin{remark}\label{rem:mfAFWsmallerThanCf}
For all pairs of functions $f$ and compact domains $\domain$, it holds that $\strongConvAFW \le \Cf$ (and $\Cf \le \CfAFW$).
\end{remark}\vspace{-2mm}
\begin{proof}
Let $\x$ be a vertex of $\domain$, so that $\S = \{\x\}$. Then $\x = \vv_f(\x)$. Pick $\x^* := \s_f(\x)$ and substitute in the definition for~$\strongConvAFW$~\eqref{eq:muf}. Then $\stepsize^\away(\x,\x^*) = 1$ and so we have $\y := \x^* = \x + \stepsize (\x^*-\x)$ with $\stepsize = 1$ which can also be used in the definition of~$\Cf$~\eqref{eq:Cf}. Thus, we have $\strongConvAFW \leq 2\left( f(\y)-f(\x) - \langle \nabla f(\x), \y-\x \rangle \right) \leq \Cf$.
\end{proof}

We now give the global linear convergence rates for the four variants of the
FW algorithm: away-steps FW (AFW Algorithm~\ref{alg:AFW}); pairwise FW
(PFW Algorithm~\ref{alg:PFW}); fully-corrective FW (FCFW
Algorithm~\ref{alg:FCFW} with approximate correction as per Algorithm~\ref{alg:correction}); 
and Wolfe's min-norm point algorithm (Algorithm~\ref{alg:FCFW} with
MNP-correction given in Algorithm~\ref{alg:MNP}). For the AFW, MNP  and PFW algorithms, we call
a \emph{drop step} when the active set shrinks, i.e. $|S^{(t+1)}| < |S^{(t)}|$. For
the PFW algorithm, we also have the possibility of a \emph{swap step}, where
$\stepsize_t = \stepmax$ but the size of the active set stays constant $|S^{(t+1)}| = |S^{(t)}|$ 
(i.e. the mass gets
fully swapped from the away atom to the FW atom). We note that a nice
property of the FCFW variant is that it does not have any drop steps (it executes both FW
steps and away steps simultaneously while guaranteeing enough progress at
every iteration).

\begin{theorem}\label{thm:megaConvergenceTheorem2}
Suppose that $f$ has smoothness constant $\CfAFW$ ($\Cf$ for FCFW and
MNP),
as well as geometric strong convexity constant~$\strongConvAFW$ as defined
in~\eqref{eq:muf}.
Then the suboptimality $h_t := f(\x^{(t)}) - f(\x^*)$ of the iterates of 
all the four variants of the FW algorithm %
decreases geometrically at each step that is not a drop step nor a swap step (i.e.
when $\stepsize_t < \stepmax$, called a `good step'\footnote{Note that any step with $\stepmax \geq 1$ can also
be considered a `good step', even if $\stepsize_t = \stepmax$, as is apparent from the proof. The problematic
steps arise only when $\stepmax \ll 1$.}), that is
\[
h_{t+1} \leq \left(1-\rho_f\right) h_t \ ,\vspace{-2mm}
\]
where:
\begin{align*}
	\rho_f &:= \frac{\strongConvAFW}{4\CfAFW} \quad \text{for the AFW algorithm,}  & 
		\rho_f &:= \min\left\{\frac{1}{2}, \frac{\strongConvAFW}{\CfAFW} \right\} \quad \text{for the PFW algorithm,}\\
	\rho_f &:= \frac{\strongConvAFW}{4\Cf} \quad \text{for the FCFW algorithm,} &
		\rho_f &:= \min\left\{\frac{1}{2}, \frac{\strongConvAFW}{\Cf} \right\} \quad \text{\parbox{0.3\linewidth}{for the MNP algorithm, or\\ FCFW with exact correction.}} 
\end{align*}

Moreover, the number of drop steps up to iteration $t$ is bounded by $t/2$.
This yields the global linear convergence rate of $h_t \leq h_0 \exp(-\frac12
\rho_f t)$ for the AFW and MNP variants. FCFW does not need the extra $1/2$
factor as it does not have any bad step. Finally, the PFW algorithm has at
most $3 |\Vertices| !$ swap steps between any two `good steps'.

If $\strongConvAFW = 0$ (i.e. the case of general convex objectives), then all the four variants have
a $O(1/k(t))$ convergence rate where $k(t)$ is the number of `good steps' up to 
iteration $t$. More specifically, we can summarize the suboptimality bounds
for the four variants as:
$$
h_t \leq \frac{4 C}{k(t) + 4}   \quad \text{for $k(t) \geq 1$},
$$
where $C = 2 \strongConvAFW + h_0$ for AFW; $C = 2 \Cf + h_0$ for 
FCFW with approximate correction; $C = \Cf/2$ for MNP; and $C=\CfAFW/2$
for PFW. The number of good steps is $k(t) = t$ for FCFW; it is $k(t) \geq t/2$ for
MNP and AFW; and $k(t) \geq t/(3|\Vertices|!+1)$ for PFW.
\end{theorem}

\begin{proof}
\textbf{Proof for AFW.} 
The general idea of the proof is to use the definition of the geometric
strong convexity constant to upper bound $h_t$, while using the definition of
the curvature constant $\CfAFW$ to lower bound the decrease in primal
suboptimality $h_t - h_{t+1}$ for the `good steps' of
Algorithm~\ref{alg:AFW}. Then we upper bound the number of `bad steps' (the
drop steps).

\emph{Upper bounding $h_t$.} In the whole proof, we assume that $\x^{(t)}$ is
not already optimal, i.e. that $h_t > 0$. If $h_t = 0$, then because
line-search is used, we will have $h_{t+1} \leq h_t = 0$ and so the geometric
rate of decrease is trivially true in this case.
Let $\x^*$ be an optimum point (which is not
necessarily unique). As $h_t > 0$, we have that $\innerProd{-\nabla
f(\x^{(t)})}{\x^*-\x^{(t)}} > 0$. We can thus apply the geometric strong
convexity bound~\eqref{eq:muf} at the current iterate $\x:=\x^{(t)}$ using
$\x^*$ as an optimum reference point to get (with $\overline{\stepsize} :=
\stepsize^\away(\x^{(t)}, \x^*)$ as defined in \eqref{eq:gammaAway}):
\begin{align}
\frac{{\overline{\stepsize}}^2}{2} \strongConvAFW  &\leq
f(\x^*)-f(\x^{(t)}) +\left\langle -\nabla f(\x^{(t)}) , \x^*-\x^{(t)} \right\rangle  \label{eq:ht_bound_trick} \\
&= -h_t + \overline{\stepsize} \left\langle - \nabla f(\x^{(t)}), \s_f(\x^{(t)}) - \vv_f(\x^{(t)})\right\rangle \nonumber\\
&\le -h_t + \overline{\stepsize} \left\langle - \nabla f(\x^{(t)}), \s_t - \vv_t \right\rangle  \nonumber\\
&=  -h_t + \overline{\stepsize}  g_t \ , \nonumber
\end{align}
where we define $g_t :=  \left\langle  -\nabla f(\x^{(t)}), \s_t - \vv_t \right\rangle$ (note that $h_t \leq g_t$ and so $g_t$ also gives a primal suboptimality certificate). %
For the third line, we have used the definition of $\vv_f(\x)$ which implies $\left\langle  \nabla f(\x^{(t)}), \vv_f(\x^{(t)})\right\rangle \leq \left\langle  \nabla f(\x^{(t)}),\vv_t \right\rangle$. %
Therefore $h_t \le -\frac{{\overline{\stepsize}}^2}{2} \strongConvAFW + \overline{\stepsize} g_t$,
which is always upper bounded\footnote{%
Here we have used the trivial inequality $0 \le %
a^2-2ab+b^2$ for the choice of numbers $a:=\frac{g_t}{\strongConvAFW}$ and $b:=\overline{\stepsize}$.
An alternative way to obtain the bound is to look at the unconstrained maximum of the RHS which is 
a concave function of $\overline{\stepsize}$ by letting $\overline{\stepsize} = g_t / \strongConvAFW$,
as we did in the main paper to obtain the upper bound on $h_t$ in~\eqref{eq:linearProgressBound}.
} %
by\vspace{-2mm}
\begin{equation} \label{eq:hUpperBound}
h_t \leq  \frac{{g_t}^2}{2\strongConvAFW}.
\end{equation}

\emph{Lower bounding progress $h_t-h_{t+1}$.} We here use
the key aspect in the proof that we had described in the main text with~\eqref{eq:AFWgapInequality}.
Because of the way the direction $\dd_t$ is chosen in the AFW
Algorithm~\ref{alg:AFW}, we have 
\begin{equation} \label{eq:gapDirection}
	\left\langle  -\nabla f(\x^{(t)}), \dd_t \right\rangle \geq g_t/2 , %
\end{equation}
and thus $g_t$ characterizes the quality of the direction $\dd_t$. To see
this, note that $2 \left\langle  \nabla f(\x^{(t)}), \dd_t \right\rangle \leq
 \left\langle \nabla f(\x^{(t)}), \dd_t^\FW\right\rangle  + \left\langle
\nabla f(\x^{(t)}), \dd_t^\away\right\rangle = \left\langle \nabla
f(\x^{(t)}), \dd_t^\FW+\dd_t^\away\right\rangle = -g_t$.

We first consider the case $\stepsize_\textrm{max} \geq 1$. Let $\x_\stepsize
:=  \x^{(t)} + \stepsize \dd_t$ be the point obtained by moving with
step-size $\stepsize$ in direction $\dd_t$, where $\dd_t$ is the one chosen
by Algorithm~\ref{alg:AFW}. By using $\s := \x^{(t)} + \dd_t$ (a feasible
point as $\stepsize_\textrm{max} \geq 1$)%
, $\x := \x^{(t)}$ and $\y := \x_\stepsize$ in the definition of the
curvature constant~$\Cf$~\eqref{eq:Cf}, and solving for $f(\x_\stepsize)$, we
get the affine invariant version of the descent lemma~\eqref{eq:descentLemma}:
\begin{equation} \label{eq:descentLemmaCf}
f(\x_\stepsize) \leq f(\x^{(t)}) + \stepsize \left\langle  \nabla
f(\x^{(t)}), \dd_t \right\rangle + \frac{\stepsize^2}{2} \Cf, \quad \text{valid $\forall
\stepsize \in [0,1]$}.
\end{equation}
As~$\stepsize_t$ is obtained by line-search and that
$[0,1] \subseteq [0,\stepsize_\textrm{max}]$, we also have that
$f(\x^{(t+1)}) = f(\x_{\stepsize_t}) \leq  f(\x_\stepsize)$ $\forall
\stepsize \in [0,1]$. Combining these two inequalities, subtracting $f(\x^*)$
on both sides, and using $\Cf \leq \CfAFW$ to simplify the possibilities
yields $h_{t+1} \le h_t + \stepsize \left\langle  \nabla f(\x^{(t)}), \dd_t
\right\rangle + \frac{\stepsize^2}{2} \CfAFW$.

Using the crucial gap inequality~\eqref{eq:gapDirection}, we get $h_{t+1} \le
h_t - \stepsize \frac{g_t}{2} + \frac{\stepsize^2}{2} \CfAFW$, and so:
\begin{equation} \label{eq:hProgress}
h_t - h_{t+1} \ge \stepsize \frac{g_t}{2} - \frac{\stepsize^2}{2} \CfAFW
\quad \forall \stepsize \in [0,1].  
\end{equation}
We can minimize the bound~\eqref{eq:hProgress} on the right hand side by
letting $\stepsize = \stepbound_t := \frac{g_t}{2\CfAFW}$. Supposing that
$\stepbound_t \leq 1$, we then get $h_t - h_{t+1} \geq 
\frac{g_t^2}{8\CfAFW}$ (we cover the case $\stepbound_t > 1$ later).  By
combining this inequality with the one from geometric strong
convexity~\eqref{eq:hUpperBound}, we get 
\begin{equation} \label{eq:mainGeometric}
h_t - h_{t+1} \ge \frac{\strongConvAFW}{4\CfAFW} \, h_t
\end{equation}
implying that we have a geometric rate of decrease $h_{t+1} \leq
\Big(1-\frac{\strongConvAFW}{4\CfAFW}\Big) h_t$ (this is a `good step').

\emph{Boundary cases.} We now consider the case $\stepbound_t > 1$ (with
$\stepsize_\textrm{max} \geq 1$ still). The condition $\stepbound_t > 1$ then
translates to $g_t \geq 2 \CfAFW$, which we can use in~\eqref{eq:hProgress}
with $\stepsize = 1$ to get $h_t - h_{t+1} \geq \frac{g_t}{2} - \frac{g_t}{4}
= \frac{g_t}{4}$. Combining this inequality with $h_t \leq g_t$ gives the
geometric decrease $h_{t+1} \leq \left(1-\frac{1}{4}\right) h_t$ (also a
`good step'). $\rho_f^\away$ is obtained by considering the worst-case of the
constants obtained from $\stepbound_t > 1$ and $\stepbound_t \leq 1$. (Note
that $\strongConvAFW \leq \CfAFW$ by Remark~\ref{rem:mfAFWsmallerThanCf},
and thus $\frac{1}{4} \geq \frac{\strongConvAFW}{4\CfAFW}$).

Finally, we are left with the case that $\stepsize_\textrm{max} < 1$. This is
thus an away step and so $\dd_t = \dd_t^\away = \x^{(t)} - \vv_t$. Here, we
use the away version $\CfAFW$: by letting $\s := \x^{(t)}$, $\vv = \vv_t$ and $\y := \x_\stepsize$ in~\eqref{eq:CfAFW}, we also
get the bound $f(\x_\stepsize) \leq f(\x^{(t)}) + \stepsize \left\langle 
\nabla f(\x^{(t)}), \dd_t \right\rangle + \frac{\stepsize^2}{2} \CfAFW$,
valid $\forall \stepsize \in [0,1]$ (but note here that the points
$\x_\stepsize$ are not feasible for $\stepsize > \stepsize_\textrm{max}$ --
the bound considers some points outside of $\domain$). We now have two
options: either $\stepsize_t = \stepsize_\textrm{max}$ (a drop step) or
$\stepsize_t < \stepsize_\textrm{max}$. In the case $\stepsize_t <
\stepsize_\textrm{max}$ (the line-search yields a solution in the interior of
$[0,\stepsize_\textrm{max}]$), then because $f(\x_\stepsize)$ is convex in
$\stepsize$, we know that $\min_{\stepsize \in [0,\stepsize_\textrm{max}]}
f(\x_\stepsize) = \min_{\stepsize \geq 0} f(\x_\stepsize)$ and thus
$\min_{\stepsize \in [0,\stepsize_\textrm{max}]} f(\x_\stepsize) =
f(\x^{(t+1)}) \leq f(\x_\stepsize)$ $\forall \stepsize \in [0,1]$. We can
then re-use the same argument above equation~\eqref{eq:hProgress} to get the
inequality~\eqref{eq:hProgress}, and again considering both the case
$\stepbound_t \leq 1$ (which yields inequality~\eqref{eq:mainGeometric}) and
the case $\stepbound_t > 1$ (which yields $(1-\frac{1}{4})$ as the geometric
rate constant), we get a `good step' with $1-\rho_f$ as the worst-case
geometric rate constant.

Finally, we can easily bound the number of drop steps possible up to
iteration $t$ with the following argument (the drop steps are the `bad steps'
for which we cannot show good progress). Let $A_t$ be the number of steps
that added a vertex in the expansion (only standard FW steps can do this) and
let $D_t$ be the number of drop steps. We have that $|\Coreset^{(t)}| =
|\Coreset^{(0)}| + A_t - D_t$. Moreover, we have that $A_t+D_t \leq t$. We
thus have $1 \leq |\Coreset^{(t)}| \leq |\Coreset^{(0)}| + t - 2D_t$,
implying that $D_t \leq \frac{1}{2}( |\Coreset^{(0)}|-1+t)=\frac{t}{2}$, as stated in the theorem.

\paragraph{Proof for FCFW.}
In the case of FCFW, we do not need to consider away steps: by the quality of the 
approximate correction in Algorithm~\ref{alg:correction} (as specified in Line~4), we know that at the
beginning of a new iteration, the away gap $g_t^\away \leq \epsilon$.
Supposing that the algorithm does not exit at line~6 of
Algorithm~\ref{alg:FCFW}, then $g_t^\FW > \epsilon$ and thus we have that $2
\innerProdCompressed{\r_t}{\dd_t^\FW} \geq
\innerProdCompressed{\r_t}{\dd_t^\PFW}$ using a similar argument as
in~\eqref{eq:AFWgapInequality} (i.e. if one would be to run the AFW algorithm at this point, it
would take a FW step). Finally, by property of the line~3 of the approximate correction
Algorithm~\ref{alg:correction}, the correction is
guaranteed to make at least as much progress as a line-search in direction
$\dd_t^\FW$, and so the lower bound~\eqref{eq:hProgress} can be
used for FCFW as well (but using $\Cf$ as the constant instead of $\CfAFW$ given that it was a FW step).

\paragraph{Proof for MNP.}
After a correction step in the MNP algorithm, we have that the current iterate is the minimizer over the active set, and thus $g_t^\away = 0$. We thus have $\innerProdCompressed{\r_t}{\dd_t^\FW} = \innerProdCompressed{\r_t}{\dd_t^\PFW} = g_t$, which means that a standard FW step would yield a geometric decrease of error.\footnote{Moreover, as we do not have the factor of $2$ relating $\innerProdCompressed{\r_t}{\dd_t^\FW}$ and $g_t$ unlike in the AFW and approximate FCFW case,
we can remove the factor of~$\frac{1}{2}$ in front of~$g_t$ in~\eqref{eq:hProgress}, removing
the factor of~$\frac{1}{4}$ appearing in~\eqref{eq:mainGeometric},
and also giving a geometric decrease with factor $(1-\frac{1}{2})$ when $\stepbound_t > 1$.}
It thus remains to show that the MNP-correction is making as much progress as a FW line-search. Consider $\y_1$ as defined in Algorithm~\ref{alg:MNP}. If it belongs to $\conv(\VVertices^{(0)})$, then it has made more progress than a FW line-search as $\s_t$ and $\x^{(t)}$ belongs to $\conv(\VVertices^{(0)})$. 

The next possibility is the crucial step in the proof: suppose that exactly one atom was removed from the correction polytope and that $\y_1$ does not belong to $\conv(\VVertices^{(0)})$ (as this was covered in the above case). This means that $\y_2$ belongs to \emph{the relative interior} of $\conv(\VVertices^{(1)})$. Because $\y_2$ is by definition the affine minimizer of $f$ on $\conv(\VVertices^{(1)})$, the negative gradient $-\nabla f(\y_2)$ is pointing away to the polytope $\conv(\VVertices^{(1)})$ (by the optimality condition). But $\conv(\VVertices^{(1)})$ is a \emph{facet} of $\conv(\VVertices^{(0)})$, this means that $-\nabla f(\y_2)$ determines a facet of $\conv(\VVertices^{(0)})$ (i.e. $\innerProdCompressed{-\nabla f(\y_2)}{\y - \y_2} \leq 0$ for all $\y \in \conv(\VVertices^{(0)})$). This means that $\y_2$ is also the minimizer of $f$ on $\conv(\VVertices^{(0)})$ and thus has made more progress than a FW line-search.

In the case that two atoms are removed from $\conv(\VVertices^{(0)})$, we cannot make this argument anymore
(it is possible that $\y_3$ makes less progress than a FW line-search); but in this case, the size of the active set is reduced by one (we have a drop step), and thus we can use the same argument as in the AFW algorithm to bound the number of such steps.

\paragraph{Proof for PFW.}
In this case, $\innerProdCompressed{\r_t}{\dd_t} = \innerProdCompressed{\r_t}{\dd_t^\PFW}$, so we do not even need a factor of 2 to relate the gaps (with the same consequence as in MNP in getting slightly bigger constants). We can re-use the same argument as in the AFW algorithm to get a geometric progress when $\stepsize_t < \stepmax$. When $\stepsize_t = \stepmax$ we can either have a drop step if $\s_t$ was already in $\Coreset^{(t)}$, or a swap step if $\s_t$ was also added to $\Coreset^{(t)}$ and so $|\Coreset^{(t+1)}| = |\Coreset^{(t)}|$. The number of drop steps can be bounded similarly as in the AFW algorithm. On the other hand, in the worst case, there could be a very large number of swap steps. We provide here a very loose bound, though it would be interesting to use other properties of the objective to prove that this worst case scenario cannot happen.

We thus bound the maximum number of swap steps between two `good steps' (very loosely). Let $m = |\Vertices|$ be the number of possible atoms, and let $r$ be the size of the current active set $|\Coreset^{(t)}| = r \leq m$. When doing a drop step $\stepsize_t = \alpha_{\vv_t}$, there are two possibilities: either we move all the mass from $\vv_t$ to a new atom $\s_t \notin \Coreset^{(t)}$ i.e. $\alpha^{(t+1)}_{\vv_t} = 0$ and $\alpha^{(t+1)}_{\s_t} = \alpha^{(t)}_{\vv_t}$ (a swap step); or we move all the mass from $\vv_t$ to an old atom $\s_t \in \Coreset^{(t)}$ i.e. $\alpha^{(t+1)}_{\s_t} = \alpha^{(t)}_{\s_t}+ \alpha^{(t)}_{\vv_t}$ (a `full drop step'). When doing a swap step, the set of 
\emph{possible values} for the coordinates $\alpha_{\vv}$ \emph{do not change}, they are only `swapped around' amongst the $m$ possible slots. The maximum number of possible consecutive swap steps without revisiting an iterate already seen is thus bounded by the number of ways we can assign $r$ numbers in $m$ slots (supposing the $r$ coordinates were all distinct in the worst case), which is $m! / (m-r)!$. Note that because the algorithm does a line-search in a strict descent direction at each iteration, we always have $f(\x^{(t+1)}) < f(\x^{(t)})$ unless $\x^{(t)}$ is already optimal. This means that the algorithm cannot revisit the same point unless it has converged. When doing a `full drop step', the set of coordinates changes, but the size of the active set is reduced by one (thus $r$ reduced by one). In the worst case, we will do a maximum number of swap steps, followed by a full drop step, repeated so on all the way until we reach an active set of only one element (in which case there is a maximum number of $m$ swap steps). Starting with an active set of $r$ coordinates, the maximum number of swap steps $B$ without doing any `good step' (which would also change the set of coordinates), is thus upper bounded by:
\begin{equation*}
B \leq \sum_{l = 1}^{r} \frac{m!}{(m-l)!} \leq m! \sum_{l=0}^{\infty} \frac{1}{l!} = m! \, e \leq 3 m! \,\, , 
\end{equation*}
as claimed.

\paragraph{Proof of Sublinear Convergence for General Convex Objectives (i.e. when $\strongConvAFW=0$).} 
For the good steps 
of the MNP algorithm and the pairwise FW algorithm, we have the reduction of suboptimality
given by~\eqref{eq:hProgress} without the factor of $\frac{1}{2}$ in front of $g_t \geq h_t$.
This is the standard recurrence that appears in the convergence proof of Frank-Wolfe
(see for example Equation~(4) in~\citep[proof of Theorem~1]{Jaggi:2013wg}), yielding
the usual convergence:
\begin{equation} \label{eq:sublinearMNP}
h_t \leq \frac{2 C}{k(t) + 2}  \quad \text{ for $k(t) \geq 1$},
\end{equation} 
where $k(t)$ is the number of good steps up to iteration $t$, and $C = \Cf$ for MNP and $C = \CfAFW$ for PFW.
The number of good steps for MNP is $k(t) \geq t/2$, while for PFW, we have the (useless) lower bound
$k(t) \geq t/(3|\Vertices|!+1)$. For FCFW with exact correction, the rate~\eqref{eq:sublinearMNP} was 
already proven in~\citep{Jaggi:2013wg} with $k(t) = t$. On the other hand, for FCFW with approximate correction, and for AFW, the factor of $\frac{1}{2}$ in front of the gap~$g_t$ in the suboptimality bound~\eqref{eq:hProgress}
somewhat complicates the convergence proof. The recurrence we get for the
suboptimality is the same as in Equation~(20) of~\citep[proof of Theorem~C.1]{LacosteJulien:2013ue},
with $\nu = \frac{1}{2}$ and $n=1$, giving the following suboptimality bound:
\begin{equation} \label{eq:sublinearAFW}
h_t \leq \frac{4 C}{k(t) + 4}  \quad \text{ for $k(t) \geq 0$},
\end{equation} 
where $C = 2\CfAFW + h_0$ for AFW and $C = 2\Cf + h_0$ for FCFW with approximate correction.
Moreover, the number of good steps is $k(t) \geq t/2$ for AFW, and $k(t) = t$ for FCFW.
A weaker (logarithmic) dependence on the initial error~$h_0$ can also be obtained by following a tighter
analysis (see~\citep[Theorem~C.4]{LacosteJulien:2013ue} or~\citep[Lemma~D.5 and Theorem~D.6]{Krishnan:2015ws}),
though we only state the simpler result here.
\end{proof}

\section{Empirical Tightness of Linear Rate Constant} \label{app:triangle}

We describe here a simple experiment to test how tight the 
constant in the linear convergence rate of Theorem~\ref{thm:megaConvergenceTheorem2}
is. We test both AFW and PFW on the triangle domain
with corners at the locations
$(-1,0)$, $(0,0)$ and $(\cos(\theta),\sin(\theta))$, for
increasingly small $\theta$ (see Figure~\ref{fig:triangle}).

\begin{figure}[t]
\begin{center}
\includegraphics[width=0.7\linewidth]{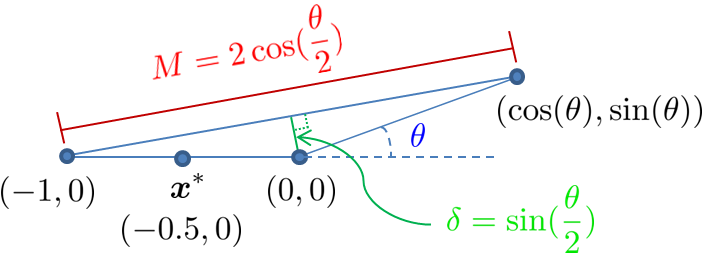}
\end{center}
\caption{Simple triangle domain to test the empirical
tightness of the constant in the convergence rate for AFW and PFW.
The width $\delta$ varies with $\theta$. We optimize $f(\x) := \frac{1}{2} \|\x-\x^*\|^2$
over this domain.} \label{fig:triangle}
\end{figure}

The pyramidal width $\delta=\sin(\frac{\theta}{2})$
becomes vanishingly small as $\theta \to 0$; the
diameter is $M=2\cos(\frac{\theta}{2})$.
We consider the optimization of the 
function $f(\x) := \frac{1}{2} \|\x-\x^*\|^2$
with $\x^*=(-0.5,1)$ on one edge of the domain.
Note that the condition number of $f$ is $\frac{L}{\mu}=1$.
The bound on the linear convergence rate $\rho$ according
to Theorem~\ref{thm:megaConvergenceTheorem2} 
(using $\CfAFW \leq L M^2$~\eqref{eq:CfBound} and $\strongConvAFW \geq \mu\, \delta^2$~\eqref{eq:muDirWidthBound})
is $\rho^\PFW = \frac{\strongConvAFW}{\CfAFW} \geq \frac{\mu}{L}\left(\frac{\delta}{M}\right)^2$
for PFW and $\rho^\away = \frac{1}{4} \rho^\PFW $ for AFW.
The theoretical constant here is thus $\rho^\PFW = \frac{1}{4} \tan^2 (\frac{\theta}{2})$.
We consider $\theta$ varying from $\pi/4$ to $1e\!-\!3$, and thus theoretical rates
varying on a wide range from $0.04$ to $1e\!-\!7$. We compare the theoretical rate $\rho$
with the empirically observed one by estimating $\hat{\rho}$ in the relationship
$h_t \approx h_0 \exp(-\rho t)$ (using linear regression on the semilogarithmic scale).
For each $\theta$, we run both AFW and PFW for 2000 iterations 
starting from 20 different random starting points\footnote{$\x_0$ is obtained by taking
a random convex combination of the corners of the domain.} 
in the interior of the triangle domain. We disregard the starting points
that yield a drop step (as then the algorithm converges in one iteration;
these happen for about $10\%$ of the starting points). Note
that as there is no drop step in our setup, we do not
need to divide by two the effective rate as is done in 
Theorem~\ref{thm:megaConvergenceTheorem2} (the number of `good steps' is $k(t) = t$).

\begin{figure}[t]
\begin{center}
\begin{minipage}{0.46\linewidth}
\centering
\includegraphics[width=\linewidth]{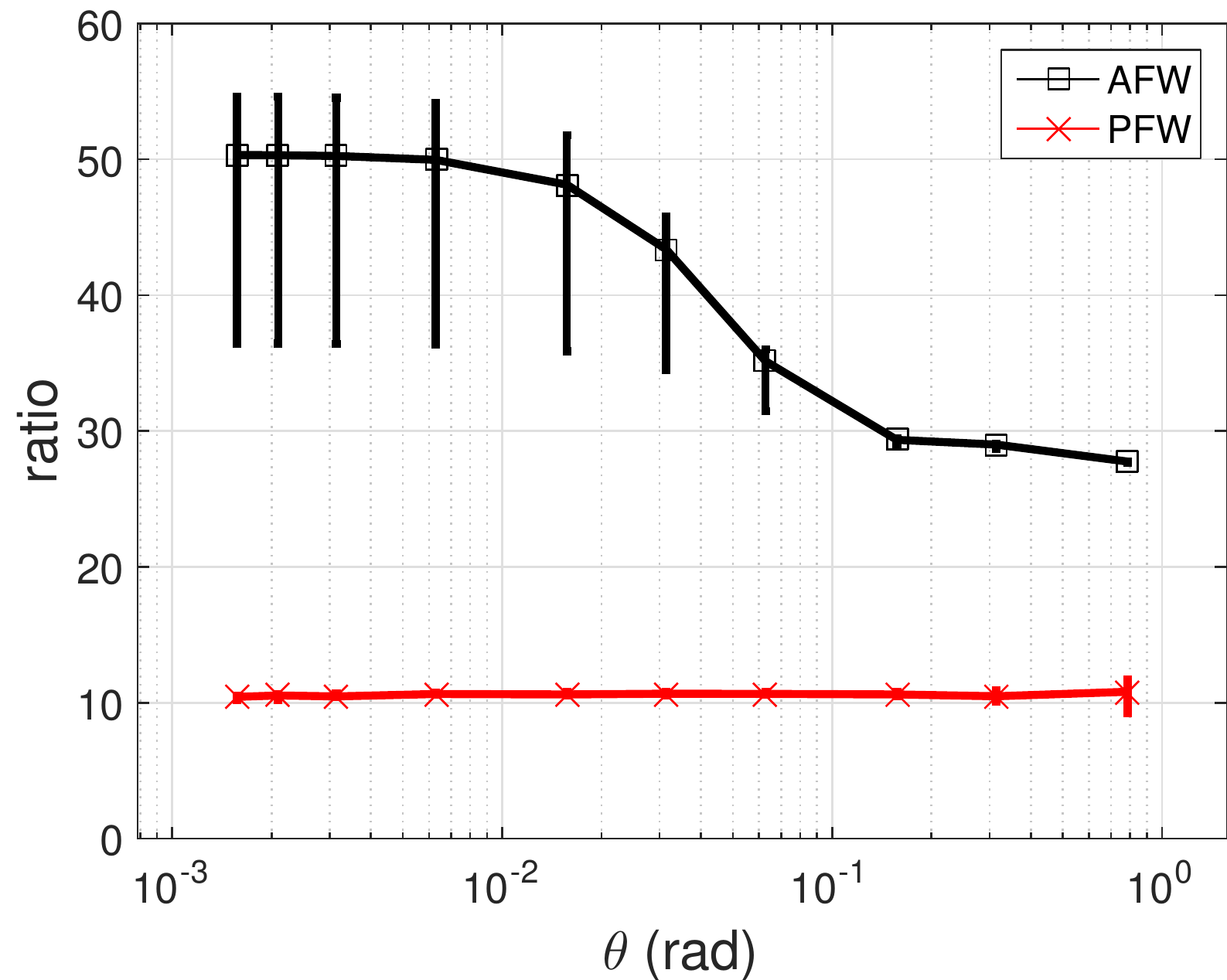}
(a)  Ratio of empirical rate vs. theoretical one
\end{minipage}	
\hfill
\begin{minipage}{0.48\linewidth}
\centering
\includegraphics[width=0.9\linewidth]{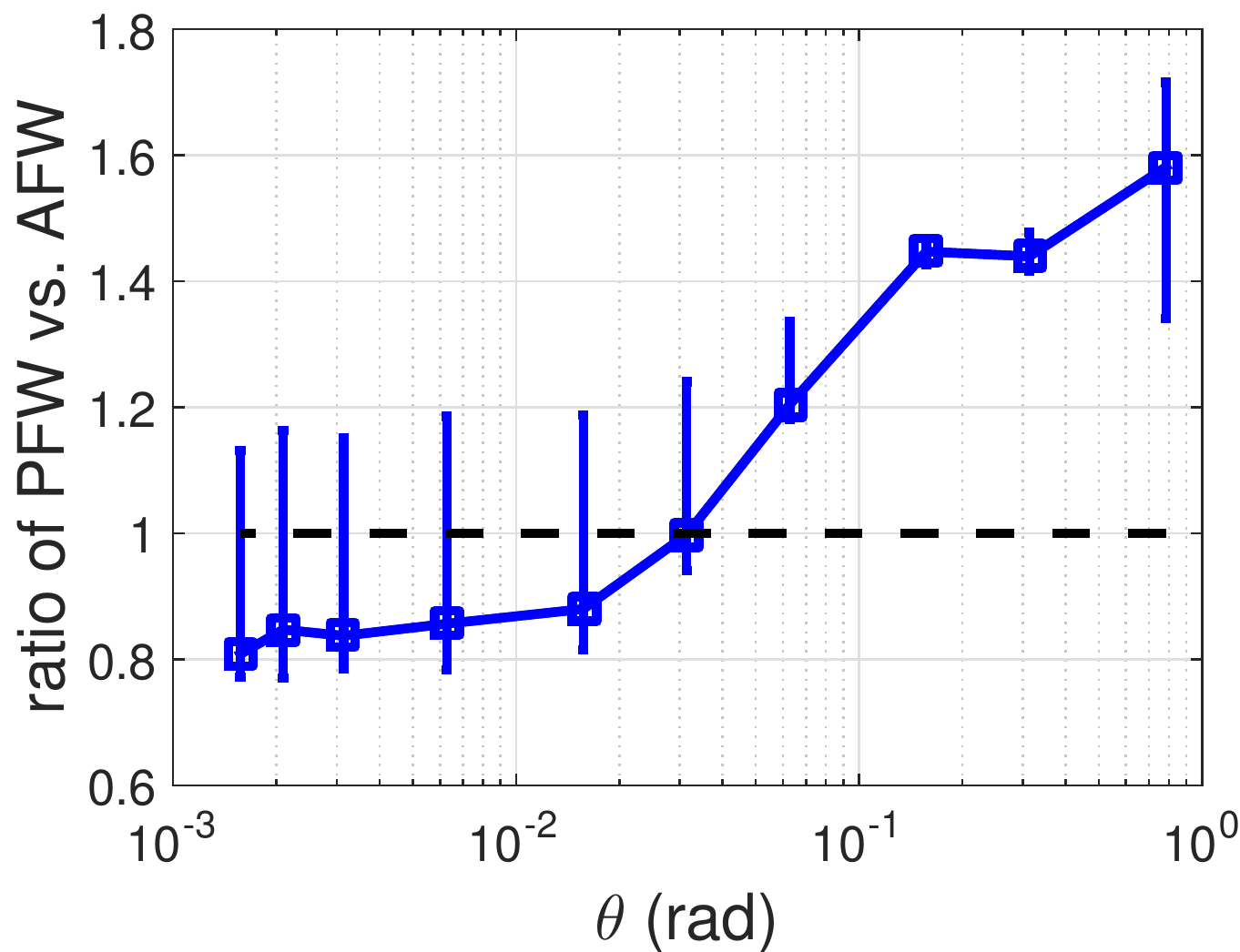}
\\[3mm]
(b) Ratio of empirical rate of PFW vs. AFW
\end{minipage}
\end{center}
\caption{Empirical rate results for the triangle domain of Figure~\ref{fig:triangle}.
We plot median values over 20 random starting points; the
error bars represent the $25\%$ and $75\%$ quantiles.
The empirical rate for PFW is closely following the theoretical one.} \label{fig:triangleResults}
\end{figure}

Figure~\ref{fig:triangleResults} presents the results.
In Figure~\ref{fig:triangleResults}(a), we
plot the ratio of the estimated rate over
the theoretical rate $\frac{\hat{\rho}}{\rho}$ for both PFW 
and AFW as $\theta$ varies. Note that the ratio
is very stable for PFW (around 10), despite the rate
changing through six orders of magnitude, demonstrating
the empirical tightness of the constant for this domain. 
The ratio for AFW has more fluctuations, but also stays
within a stable range. We can also 
do a finer analysis than the pyramidal width and
consider the finite number possibilities for the worst case angles
 for $(\innerProdCompressed{\hat{\r}}{{\dd^\PFW(\r)}})^2$.
This gives the tighter constant $\rho^\PFW = \sin^2(\frac{\theta}{2})$
for our triangle domain, gaining a factor of about~4, but still
not matching yet the empirical observation for PFW.

In Figure~\ref{fig:triangleResults}(b), we compare the empirical
rate for PFW vs. the one for AFW. For bigger theoretical rates, PFW appears
to converge faster. However, AFW gets a slightly
better empirical rate for very small rates (small angles).

\section{Non-Strongly Convex Generalization} \label{app:NonStronglyConvex}

Here we will study the generalized setting with objective $f(\x) := g(\A \x) + \innerProdCompressed{\bv}{\x}$ where $g$ is $\mu_g$-\emph{strongly convex} w.r.t.\ the Euclidean norm over the domain $\A \domain$ with strong convexity constant $\mu_g > 0$. 

We first define a few constants: let $G := \max_{\x \in \domain} \| \nabla g
(\A \x) \|$ be the maximal norm of the gradient of $g$ over $\A \domain$; $M$
be the diameter of $\domain$ and $M_{\A}$ be the diameter of $\A \domain$.

Let $\theta$ be the Hoffman constant (see~\citep[Lemma 2.2]{Beck:2015vo})
associated with the matrix $[\A ; \bv^\top; \B] = \left(\substack{\A \\ \bv^\top \\ \B}\right)$, where the rows of $\B$ are the linear inequality constraints defining the set $\domain$.

We present here a generalization of Lemma 2.5 from~\citep{Beck:2015vo}:
\begin{lemma}\label{lem:generalizedStrongConvexity}
For any $\x \in \domain$ and $\x^*$ in the solution set $\X^*$:
\begin{equation}
f(\x^*)- f(\x)- 2 \langle\nabla f(\x), \x^*-\x \rangle \geq 2 \tilde{\mu} \, d(\x, \X^*)^2 ,
\end{equation}
where $\tilde{\mu} := 1/\left(2 \theta^2 \left( \|\bv \| M + 3 G M_{\A} + \frac{2}{\mu_g} (G^2+1) \right) \right)$ is the generalized strong convexity constant for $f$.
\end{lemma}
\begin{proof}
Let $\x^*$ be any element of the solution set $\X^*$. By the strong convexity of $g$, we have
\begin{equation} \label{eq:strongConvForG}
f(\x^*)- f(\x)- \langle\nabla f(\x), \x^*-\x \rangle
\geq \textstyle\frac{\mu_g}{2} \norm{\A \x^* - \A \x}^2 \ .
\end{equation}

Moreover, by the convexity of $f$, we have:
\begin{equation} \label{eq:convexityForG}
- \langle\nabla f(\x), \x^*-\x \rangle
\geq f(\x) - f(\x^*) .
\end{equation}

We now use inequality (2.10) in~\citep{Beck:2015vo} to get the bound:
\begin{equation} \label{eq:lowerBoundWithB}
f(\x) - f(\x^*) \geq  \frac{1}{B_1} (\innerProdCompressed{\bv}{\x} - \innerProdCompressed{\bv}{\x^*} )^2 ,
\end{equation}
where $B_1 := ( \|\bv \| M + 3 G M_{\A} + \frac{2}{\mu_g} G^2)$.

Plugging~\eqref{eq:lowerBoundWithB} into~\eqref{eq:convexityForG} and adding to~\eqref{eq:strongConvForG}, we get:
\begin{align*} 
f(\x^*)- f(\x)- 2 \langle\nabla f(\x), \x^*-\x \rangle
&\geq \frac{1}{B_2} \left( \norm{\A \x^* - \A \x}^2  + (\innerProdCompressed{\bv}{\x} - \innerProdCompressed{\bv}{\x^*} )^2 \right) \\
&\geq \frac{1}{B_2 \theta^2} d(\x, \X^*)^2 ,
\end{align*}
where $B_2 := ( \|\bv \| M + 3 G M_{\A} + \frac{2}{\mu_g} (G^2+1))$. For the last inequality, we used 
inequality~(2.1) in~\citep{Beck:2015vo} that made use of
the Hoffman's Lemma (see~\citep[Lemma 2.2]{Beck:2015vo}),
where $\theta$ is the Hoffman constant associated with the matrix $[\A ; \bv^\top ; \B]$. In this case, $\B$ is the matrix with rows containing the linear inequality constraints defining $\domain$.
\end{proof}

We now define the following generalization of the geometric strong convexity constant \eqref{eq:muf}, that we now call $\strongConvGeneralized$:
\begin{equation}\label{eq:mufGeneralized}
  \strongConvGeneralized := \inf_{\x\in \domain} \sup_{\substack{\x^* \in \X^*\\
                        \textrm{s.t. } \left\langle \nabla f(\x), \x^*-\x \right\rangle < 0 }}
           \frac{1}{{2 \stepsize^\away(\x,\x^*)}^2}
           \big( f(\x^*)-f(\x)-2 \left\langle \nabla f(\x),  \x^*-\x \right\rangle \big) \ .
\end{equation}
Notice the new inner \emph{supremum} over the solution set $\X^*$ compared to the original definition \eqref{eq:muf}, the factor of $2$ in front of the gradient, and the different 
overall scaling to have a similar form as in the previous linear convergence theorem. 
This new quantity $\strongConvGeneralized$ is still \emph{affine invariant}, but
unfortunately now depends on the location of the solution set $\X^*$.
We now present the generalization of Theorem~\ref{thm:muFdirWinterpretation2}.

\begin{theorem}\label{thm:generalizedPWidth}
Let $f(\x) := g(\A \x) + \innerProdCompressed{\bv}{\x}$ where $g$ is $\mu_g$-\emph{strongly convex} w.r.t.\ the Euclidean norm over the domain $\A \domain$ with strong convexity constant $\mu_g > 0$. Let $\tilde{\mu}$ be the corresponding generalized strong convexity constant coming from Lemma~\ref{lem:generalizedStrongConvexity}.
Then
\[
\strongConvGeneralized \geq \tilde{\mu} \cdot \left( \PWidth(\domain) \right)^2 \ .
\] 
\end{theorem}
\begin{proof}
Let $\x$ be fixed and not optimal; let $\x^*$ be its closest point in $\X^*$ i.e. $\|\x - \x^*\| = d(\x, \X^*)$. We have that $\left\langle \nabla f(\x), \x^*-\x \right\rangle < 0$ as $\x$ is not optimal.

We use the generalized strong convexity notion $\tilde{\mu}$ from Lemma~\ref{lem:generalizedStrongConvexity} for the particular reference point~$\x^*$ in the third line below to get:
\begin{multline*}
\sup_{\substack{\x' \in \X^*\\
                        \textrm{s.t. } \left\langle \nabla f(\x), \x'-\x \right\rangle < 0 }}
           \frac{1}{{2\stepsize^\away(\x,\x')}^2}
           \big( f(\x')-f(\x)-2 \left\langle \nabla f(\x),  \x'-\x \right\rangle \big)  \\
           \geq  \frac{1}{{2\stepsize^\away(\x,\x^*)}^2}
                      \big( f(\x^*)-f(\x)-2 \left\langle \nabla f(\x),  \x^*-\x \right\rangle \big) \\ 
          \geq  \frac{1}{{2\stepsize^\away(\x,\x^*)}^2} 2 \tilde{\mu}\, d(\x, \X^*)^2
          = \frac{1}{{ \stepsize^\away(\x,\x^*)}^2} \tilde{\mu} \|\x-\x^*\|^2 .                  
\end{multline*}
We can do this for each non-optimal $\x$. We thus obtain:
\begin{equation} \label{eq:muTildeInequality}
\strongConvGeneralized \geq  \inf_{\substack{\x, \x^* \in \domain\\
                                   \textrm{s.t. } \innerProd{\r_{\x}}{\x^*-\x} > 0}}
                      \tilde{\mu} \left(  \frac{\innerProd{\r_{\x}}{ \s_f(\x) - \vv_f(\x)}}{\innerProd{\r_{\x}}{\x^*-\x}} \norm{{\x^*-\x}} \right)^2 .
\end{equation}
And we are back to the same situation as in the proof of our earlier Theorem~\ref{thm:muFdirWinterpretation2}, the only change being that we now have equation~\eqref{eq:mufInitial} holding for the general strong convexity constant $\tilde{\mu}$ instead of its classical analogue~$\mu$.
\end{proof}

Having this tool at hand, the linear convergence of all Frank-Wolfe algorithm variants now holds with the earlier $\strongConvAFW$ complexity constant replaced with~$\strongConvGeneralized$. The factor of 2 in the denominator of~\eqref{eq:mufGeneralized} is to ensure the same scaling.%

Again, as we have shown in Theorem \ref{thm:generalizedPWidth}, we have that our condition $\strongConvGeneralized > 0$ leading to linear convergence is slightly weaker than generalized strong convexity in the Hoffman sense (it is implied by it).

\begin{theorem}\label{thm:megaConvergenceTheoremHoffman}
Suppose that $f$ has smoothness constant $\CfAFW$ ($\Cf$ for FCFW and
MNP),
as well as generalized geometric strong convexity constant~$\strongConvGeneralized$ as defined
in~\eqref{eq:mufGeneralized}.

Then the suboptimality error~$h_t$ of the iterates of all the four variants of the FW algorithm (AFW, FCFW, MNP and PFW) decreases geometrically at each step that is not a drop step nor a swap step (i.e.
when $\stepsize_t < \stepmax$), with the same constants as in Theorem~\ref{thm:megaConvergenceTheorem2}, except that $\strongConvAFW$ is replaced by~$\strongConvGeneralized$.
\end{theorem}

\begin{proof}
The proof closely follows the proof of Theorem \ref{thm:megaConvergenceTheorem2}.

We start from the above generalization~\eqref{eq:mufGeneralized} of the original geometric strong convexity constant \eqref{eq:muf}, and first replace the $\inf$ over $\x$ by considering only the choice $\x := \x^{(t)}$, giving
\begin{equation}\label{eq:mufGeneralizedAgain}
  \strongConvGeneralized \le  \sup_{\substack{\x^* \in \X^*\\
                        \textrm{s.t. } \left\langle \nabla f(\x^{(t)}), \x^*-\x^{(t)} \right\rangle < 0 }}
           \frac{1}{{2 \stepsize^\away(\x^{(t)},\x^*)}^2}
           \big( f(\x^*)-f(\x^{(t)})-2 \left\langle \nabla f(\x^{(t)}),  \x^*-\x^{(t)} \right\rangle \big) \ .
\end{equation}

From here, we will now mirror our earlier derivation for an upper bound on the suboptimality 
as a function of the gap $g_t$, as given in~\eqref{eq:ht_bound_trick}.
As an optimal reference point $\x^*$ in \eqref{eq:ht_bound_trick}, we will choose a $\tilde \x^*$ attaining the supremum in \eqref{eq:mufGeneralizedAgain}, given $\x^{(t)}$.

We again employ the `step-size quantity' $\overline{\stepsize} :=
\stepsize^\away(\x^{(t)}, \tilde \x^*)$ as defined in \eqref{eq:gammaAway}. Using \eqref{eq:mufGeneralizedAgain}, we have
\begin{align}\label{eq:ht_bound_trickAgain}
   2 \overline{\stepsize}^2 %
   \,\strongConvGeneralized
    &\le  f(\x^*)-f(\x^{(t)})+2 \left\langle -\nabla f(\x^{(t)}),  \tilde \x^*-\x^{(t)} \right\rangle  \\
&= -h_t + 2\overline{\stepsize} \left\langle  -\nabla f(\x^{(t)}), \s_f(\x^{(t)}) - \vv_f(\x^{(t)})\right\rangle \nonumber\\
&\le -h_t + 2\overline{\stepsize} \left\langle  -\nabla f(\x^{(t)}), \s_t - \vv_t \right\rangle  \nonumber\\
&=  -h_t + 2\overline{\stepsize}  g_t \ , \nonumber
\end{align}

Therefore $h_t \le -\frac{{\hat{\stepsize}}^2}{2} \strongConvGeneralized + \hat{\stepsize} g_t$ when writing $\hat{\stepsize}:=2\overline{\stepsize}$,
which is always upper bounded\footnote{%
Here we have again used the trivial inequality $0 \le %
a^2-2ab+b^2$ for the choice of numbers $a:=\frac{g_t}{\strongConvGeneralized}$ and $b:=\hat{\stepsize}$.
} %
by\vspace{-2mm}
\begin{equation}
h_t \leq  \frac{{g_t}^2}{2\strongConvGeneralized}.
\end{equation}
which is exactly the bound~\eqref{eq:hUpperBound} as in the classical case, with the denominator being~$2\strongConvGeneralized$ instead of~$2\strongConvAFW$.

From here, the proof of the main convergence Theorem \ref{thm:megaConvergenceTheorem2} continues without modification, using~$\strongConvGeneralized$ instead of~$\strongConvAFW$.
\end{proof}

\bigskip
\bigskip
\bibliographystylesup{abbrvnat}
\bibliographysup{references}

\end{document}